%% file: static-solution_low-frequency-v04-siam.tex
\documentclass[a4paper]{amsart}
\usepackage[normalem]{ulem}
\usepackage{amscd}
\usepackage{verbatim}
\usepackage{caption}
\newcommand\numberingtheoremsectionyesno{yes}
\newcommand\numberingequationsectionyesno{yes}
\newcommand\pagesizeextendednormal{extended}
\newcommand\reportudemathyesno{no}
\newcommand\reportudemathnumber{SM-UDE-821}
\newcommand\reportudemathyear{2019}
\newcommand\reportudematheingang{\mydate}

\newcommand{\mytitle}{\Large Low Frequency Asymptotics and Electro-Magneto-Statics 
for Time-Harmonic Maxwell's Equations in Exterior Weak Lipschitz
Domains with Mixed Boundary Conditions}
\newcommand{\mytitlerepude}{Low Frequency Asymptotics and Electro-Magneto-Statics\\ 
for Time-Harmonic Maxwell's Equations in Exterior Weak Lipschitz\\
Domains with Mixed Boundary Conditions}
\newcommand{\myshorttitle}{Low Frequency Asymptotics
for Time-Harmonic Maxwell's Equations in Exterior Domains}
\newcommand{\myauthorone}{Frank Osterbrink}
\newcommand{\myauthortwo}{Dirk Pauly}
\newcommand{\myauthors}{\myauthorone\quad\&\quad\myauthortwo}
\newcommand{\myaddressone}{Fakult\"at f\"ur Mathematik,
Universit\"at Duisburg-Essen, Campus Essen, Germany}
\newcommand{\myaddresstwo}{Fakult\"at f\"ur Mathematik,
Universit\"at Duisburg-Essen, Campus Essen, Germany}
\newcommand{\myemailone}{frank.osterbrink@uni-due.de}
\newcommand{\myemailtwo}{dirk.pauly@uni-due.de}
\newcommand{\mykeywords}{low frequency asymptotics, exterior boundary value problems, Maxwell's equations, electro-magneto-statics, Hodge-Helmholtz decompositions, radiating solutions, Dirichlet-Neumann fields, cohomology groups, polynomial decay of eigensolutions}
\newcommand{\mysubjclass}{35Q60, 78A25, 78A30}
\newcommand{\mydate}{\today}
\newcommand{\mythanks}{Corresponding author: \myauthortwo}
\include{header_Osterbrink}

\makeatletter
\pgfdeclareradialshading[tikz@ball]{ring}{\pgfpoint{0cm}{0cm}}%
{rgb(0cm)=(1,1,1);
rgb(0.719cm)=(1,1,1);
color(0.72cm)=(tikz@ball);
rgb(0.9cm)=(1,1,1)}
\tikzoption{ring color}{\pgfutil@colorlet{tikz@ball}{#1}
\def\tikz@shading{ring}\tikz@addmode{\tikz@mode@shadetrue}}
\makeatother
\title[\sc\myshorttitle]{\Large\sf\mytitle}
\author{\myauthorone}
\author{\myauthortwo}
\address{\myaddressone}
\email[\myauthorone]{\myemailone}
\address{\myaddresstwo}
\email[\myauthortwo]{\myemailtwo}
\subjclass{\mysubjclass}
\date{\mydate}
\thanks{\mythanks}
\begin{document}
%
%
\ifthenelse{\equal{\reportudemathyesno}{yes}}
{\preprintudemath{\mytitlerepude}{\myauthors}{\reportudemathnumber}
{\reportudemathyear}{\reportudematheingang}}{}
%
%
\begin{abstract}
We prove that the time-harmonic solutions to Maxwell's equations in a 3D exterior 
domain converge to a certain static solution as the frequency tends to zero. We 
work in weighted Sobolev spaces and construct new compactly supported replacements 
for Dirichlet-Neumann fields. Moreover, we even show convergence in operator norm. 
\end{abstract}
\maketitle
\tableofcontents
%
%
\leqnomode
%
%
\section{Introduction}\label{sec:introduction}
%
%
Applying a time-harmonic ansatz (or Fourier-transformation with respect to  
time) to the classical time-dependent Maxwell equations 
in some domain $\Om\subset\rthree$, we are led to consider the 
\emph{time-harmonic Maxwell system}
\begin{align}\label{equ:int_max-sys}
	\rot E+i\om B=G\,,
	\qquad\qquad
	-\rot H+i\om D=-F\,,
	\qquad\qquad\text{in }\Om
\end{align}
with frequency $\om\in\C$. Here, $E$ and $H$ denote the electric and magnetic 
field, $D=\eps E$ and $B=\mu H$ represent the displacement current and 
magnetic induction, respectively, and $F$,\,$G$ are known source terms. The 
matrix valued functions $\eps$ and $\mu$ describe the permittivity and permeability 
of the medium filling $\Om$ and are assumed to be time-independent. In the 
following we are specifically interested in the case of an exterior weak Lipschitz 
domain $\Om\subset\rthree$ (\,i.e., a connected open subset with compact 
complement\,) with boundary $\Ga:=\partial\hspace*{0.01cm}\Om$ (\,Lipschitz 
submanifold\,) decomposed into two relatively open subsets $\Ga_1$ and 
$\Ga_2:=\Ga\sm\ovl{\Ga}_{1}$ being itself Lipschitz submanifolds of $\Ga$. We 
impose mixed homogeneous boundary conditions, which in classical terms can be 
written as 
\begin{align}\label{equ:int_bd-cond}
	n\times E=0\text{ on }\Ga_{1},
	\qquad\qquad
	n\times H=0\text{ on }\Ga_{2},
    \qquad\qquad
    (\,n:\text{ outward unit normal}\,),
\end{align}
and, in order to separate outgoing from incoming waves, we require the so called 
Silver-M\"uller radiation condition 
(with $\xi(x):=x/r(x)$ and $r(x):=|x|$ for $x\in\rthree$)
\begin{align}\label{equ:int_rad-cond}
	\xi\times H+E\,,\;\xi\times E-H=\mathrm{o}(r^{-1})
	\qquad\text{for}\qquad 
	r\To\infty\,.
\end{align}
\indent
First existence results concerning boundary value problems for 
the time-harmonic Maxwell system in exterior domains have been given by M\"uller 
\cite{muller_behavior_1954,muller_randwertprobleme_1952} in domains with smooth 
boundaries and homogeneous, isotropic media, i.e. $\eps=\mu=\mathbbm{1}$. In 
\cite{leis_zur_1968} Leis used the limiting absorption principle to obtain 
existence and uniqueness for media, which are possibly inhomogeneous and 
anisotropic within a bounded subset of $\Om$. Nevertheless, Leis still needed 
strong assumptions on the boundary regularity. In the bounded domain case, even 
for general inhomogeneous and anisotropic media (\,cf. Leis 
\cite{leis_aussenraumaufgaben_1974}\,), it is sufficient that $\Om$ allows for a 
certain selection theorem, later called \emph{Weck's selection theorem} or 
\emph{Maxwell compactness property}, which holds for a class of boundaries much 
larger than those accessible by the detour over $\ho$ (\,cf. Weck 
\cite{weck_maxwells_1974}, Weber \cite{weber_local_1980}, Picard 
\cite{picard_elementary_1984}, Costabel \cite{costabel_remark_1990}, 
Witsch \cite{witsch_remark_1993}, and Picard, Weck, and Witsch 
\cite{picard_time-harmonic_2001}\,). 
The most recent result for a solution theory in the exterior domain case is due 
to the second author \cite{pauly_low_2006, pauly_polynomial_2012} (\,see also 
\cite{pauly_niederfrequenzasymptotik_2003}\,) and in its structure comparable to 
the results of \cite{picard_time-harmonic_2001}. While all these results handle 
the case of full boundary conditions, in \cite{osterbrink_time-harmonic_2019} 
the authors treated for the first time mixed boundary conditions. Using the 
framework of polynomially weighted Sobolev spaces from 
\cite{picard_time-harmonic_2001}, we have been able to show that the time-harmonic 
boundary value problem \eqref{equ:int_max-sys}, \eqref{equ:int_bd-cond}, and 
\eqref{equ:int_rad-cond} admits unique solutions. In particular, by means of 
Eidus limiting absorption principle \cite{eidus_principle_1965} (\,see also 
\cite{eidus_spectra_1985, eidus_limiting_1986}\,) for the 
physically interesting case of real frequencies $\om$ a Fredholm alternative type 
result holds true. Similar to the bounded domain case, the crucial tool for 
existence is again a compact embedding result, now being a local version of Weck's 
selection theorem.\\
\indent
In this paper we investigate the low frequency behaviour of the corresponding 
time-harmonic solution operator. To this end we first have to provide a solution 
theory for the static boundary value problem, i.e., $\om=0$, which reads
\begin{alignat}{5}\label{equ:int_static-system}
    \rot E&=G\;\,&&\text{ in }\;\Om\,,\hspace*{2cm}& &
      \hspace*{2.5cm}& \rot H&=F\;\,&&\text{ in }\;\Om\,,\notag\\
    \div\eps E&=f\;\,&&\text{ in }\;\Om\,,\hspace*{2.5cm}& &
      \hspace*{2.5cm}& \div\mu H&=g\;\,&&\text{ in }\;\Om\,,\notag\\[-8pt]
      &&&&&\\[-8pt]
    n\times E&=0\;\,&&\text{ on }\;\Gat\,,\hspace*{2.5cm}& &
      \hspace*{2.5cm}& n\times H&=0\;\,&&\text{ on }\;\Gan\,,\notag\\
    n\cdot\eps E&=0\;\,&&\text{ on }\;\Gan\,,\hspace*{2.5cm}& &
      \hspace*{2.5cm}& n\cdot\mu H&=0\;\,&&\text{ on }\;\Gat\,.\notag
\end{alignat} 
\indent
There are two major challenges:
\begin{itemize}
	\item Problems in exterior domains require to work in weighted Sobolev spaces. 
	\item The systems \eqref{equ:int_static-system} have non trivial kernels, 
		  forcing us to work with orthogonality constraints on solutions in 
		  weighted Sobolev spaces to achieve uniqueness.	This specific difficulty 
		  is overcome by a construction of special compactly supported fields and 
		  certain functionals, see Theorem \ref{thm:stp_dir-neu-char}.  
\end{itemize}

In the case of full homogeneous boundary conditions and homogeneous, isotropic 
media Kress \cite{kress_potentialtheoretische_1972} (\,using integral equation 
methods\,) and Picard \cite{picard_randwertaufgaben_1981} (\,using Hilbert space 
methods\,) established solution theories in the generalized setting of alternating 
differential forms on Riemannian manifolds of arbitrary dimensions (\;see also 
\cite{picard_decomposition_1990} for nonlinear materials\,). For the classical 
threedimensional case of electro-magneto-statics with full homogeneous boundary 
conditions, we refer to Picard \cite{picard_boundary_1982} 
(\,see also \cite{milani_decomposition_1988}\,) as well as the references therein.
Following the Hilbert space approach, in Section \ref{sec:e-m-static} we will 
present Helmholtz type decompositions in weighted Sobolev spaces which then 
together with Weck's local selection theorem will provide a powerful setting for 
solving system \eqref{equ:int_static-system}.\\
\indent
In Section 4 we shortly present the time-harmonic solution theory summarizing the 
results obtained in \cite{osterbrink_time-harmonic_2019}. This results follow by 
the same methods as in \cite{pauly_low_2006,pauly_polynomial_2012} (\,see also 
Picard, Weck, and Witsch \cite{picard_time-harmonic_2001}, Weck and Witsch 
\cite{weck_complete_1992,weck_generalized1_1997,weck_generalized2_1997}\,). For 
nonreal frequencies the solution is obtained by standard Hilbert space methods as 
$\om$ belongs to the resolvent set of the \emph{Maxwell operator}
\begin{align*}
	\maps{\scrM}
		{\Rztom\times\Rznom\subset\ltLaom}{\ltLaom}
		{(E,H)}{i\La^{-1}\M}
		\,,\quad\quad
		\ltLaom:=\ltepsom\times\ltmuom\,, 
\end{align*}
where
\begin{align*}
	\La:=\ptwomat{\eps}{0}{0}{\mu}\,,
	\qquad\qquad
	\M:=\ptwomat{0}{-\rot}{\rot}{0}\,,
	\qquad\qquad
	\ltga(\Om):=\big(\,\ltom,\scp{\,\ga\,\cdot\,}{\,\cdot\,}_{\ltom}
	\,\big)\,.
\end{align*}
The case of real frequencies $\om\neq 0$ is much more challenging, since here we 
want to solve in the continuous spectrum of the Maxwell operator. Nevertheless, 
restricting to data $(F,G)\in\ltbigom{\frac{1}{2}}\times\ltbigom{\frac{1}{2}}$, 
we are able to obtain radiating solutions 
$(E,H)\in\ltsmom{-\frac{1}{2}}\times\ltsmom{-\frac{1}{2}}$ by means of Eidus' 
limiting absorption principle \cite{eidus_principle_1965,eidus_spectra_1985}, i.e., 
as limit of solutions corresponding to frequencies $\om\in\C_{+}\sm\reals$. In 
other words, the resolvent $(\scrM-\om)^{-1}$ and hence also 
$\calL_{\La,\om}=i(\,\scrM-\om\,)^{-1}\La^{-1}$ may be extended continuously to the 
real axis (\,cf. \cite{leis_initial_2013}\,). An a-priori-estimate and the 
polynomial decay of eigenfunctions needed in the limit process are obtained by 
transferring well known results for the Helmholtz equation in the whole space using 
a suitable decomposition of the fields $E$ and $H$ and perturbation arguments. 
This will be sufficient to show that a generalized Fredholm alternative holds, see
Theorem \ref{thm:thp_fredh-alt}. We have to admit finite dimensional eigenspaces 
for certain eigenvalues $\om\neq 0$, which can not accumulate in $\reals\sm\{0\}$. 
Next by proving an estimate for the solutions of the homogeneous and isotropic 
whole space problem together with an perturbation argument, we show that 
these possible eigenvalues do not accumulate even at $\om=0$. Therefore, for 
small $\om\neq 0$ the time-harmonic solution operator $\calL_{\La,\om}$ is well 
defined on $\ltbigom{\frac{1}{2}}\times\ltbigom{\frac{1}{2}}$ and a 
low frequency analysis is reasonable.\\
\indent
Finally, in Section \ref{sec:low-frequency-asymptotics} we investigate the 
low frequency behavior of time-harmonic solutions, in particular the 
question under which conditions radiating solutions converge to a static 
solution of system \eqref{equ:int_static-system}. In the case of a bounded 
domain the low frequency asymptotics is simply given by a Neumann series of the 
static solution operator $\calL_{0}$, which directly follows by applying 
$\calL_{0}$ to the time-harmonic system \eqref{equ:int_max-sys}. More precisely, 
in the case that $\Om$ is a bounded Lipschitz domain, by Weck's selection 
theorem the range $\calR(\scrM)$ of the Maxwell operator is closed and the reduced 
Maxwell operator
\begin{align*}
	\maps{\scrM_{\mathrm{red}}}{\calD(\scrM)\cap\calR(\scrM)\subset\calR(\scrM)}
	{\calR(\scrM)}{(E,H)}{(-i\eps^{-1}\rot H,i\mu^{-1}\rot E)}
\end{align*}
has a continuous inverse 
$\map{\calL_{0}}{\calR(\scrM)}{\calD(\scrM)\cap\calR(\scrM)}$, 
which interpreted as operator into $\calR(\scrM)$ is even compact. Moreover, 
arbitrary powers $\calL_{0}^{j}$ of $\calL_{0}$ are well defined. 
Hence, for small $|\om|>0$ the time-harmonic solution operator 
$\map{\calL_{\om}}{\ltLaom}{\calD(\scrM)}$ is well defined 
(\,Fredholm alternative\,) and is given by the Neumann series
\begin{align}\label{equ:int-lfa}
	\calL_{\om}=-\om^{-1}\pi_{\calN(\scrM)}
	+\sum_{j=0}^{\infty}\om^{j}\calL_{0}^{j+1}\pi_{\calR(\scrM)}	\,.
\end{align}
Here, $\pi_{\calN(\scrM)}$ and $\pi_{\calR(\scrM)}$ are the projections onto the 
kernel and the range of $\scrM$, respectively.\\
\indent
In the exterior domain case this simple low frequency asymptotics does not hold. It 
is even not well defined in an obvious way, since now the static solution operator 
$\calL_{0}$ maps data from a polynomially weighted Sobolev space to solutions 
belonging to a less weighted Sobolev space 
(\,cf. Theorem \ref{thm:stp_sol-thm-es} resp. Theorem \ref{thm:stp_sol-thm-ms}\,).
However, using an estimate for the solutions of the homogeneous, isotropic whole 
space problem together with a perturbation argument we can prove the convergence of 
the time-harmonic solutions $\calL_{\om}(F,G)$ to a specific static solution 
$\calL_{0}(F,G)$ on a certain subspace, i.e., 
\begin{align*}
	\calL_{\om}\To\calL_{0}\,.
\end{align*}
\indent
A proper and corrected version of the low frequency asymptotics \eqref{equ:int-lfa} 
for the case of an exterior domain will be addressed in a forthcoming publication 
(\,see \cite{pauly_complete_2008, pauly_hodgehelmholtz_2008, 
pauly_generalized_2009, pauly_low_2006, pauly_niederfrequenzasymptotik_2003} for 
the case of full boundary conditions\,).
%
%
\section{Preliminaries}\label{sec:prel}
%
%
In the following, $\Om\subset\rthree$ is an exterior weak Lipschitz 
domain, see \cite[Definition 2.3]{bauer_maxwell_2016}, with boundary 
$\Ga:=\partial\Om$ decomposed into two relatively open weak Lipschitz subdomains 
$\Gat$ and $\Gan:=\Ga\sm\ovl{\Ga}_{1}$, see 
\cite[Definition 2.5]{bauer_maxwell_2016}. For $x\in\rthree$ with $x\neq 0$ let 
$r(x):=|\,x\,|$ and $\xi(x):=x/|\,x\,|$ \big(\;$|\,\cdot\,|$\,: Euclidean norm in 
$\rthree$\,\big). Moreover, we fix $\rhat>0$ such that 
$\rthree\sm\Om\Subset\U_{\rhat}$ (compactly included) and define 
\begin{align*}
	\Omda:=\Om\cap \U_{\da}\,,
	\qquad\qquad
	\Ga_{i,\da}:=\Ga_{i}\cup\Sp_{\da}\,,
	\qquad\qquad
	\cU_{\da}:=\rthree\sm \ovl{\U}_{\da}\,,
	\qquad\qquad
	\G_{\rhat,\da}:=\cU(\rhat)\cap\U(\da),
	\qquad\;\;(\,\da\geq\rhat\,)\,,
\end{align*}
where $\U_{\da}$ and $\Sp_{\da}$ denote the open ball resp.\;sphere of radius $\da$ 
centered at the origin. We also pick some 
\begin{align*}
	\tilde{\eta}\in\ci(\reals)
	\quad\qquad\text{with}\quad\qquad
	0\leq\tilde{\eta}\leq 1\,,
	\quad
	\supp\tilde{\eta}\subset(1,\infty)\,,
	\quad
	\restr{\tilde{\eta}}{[2,\infty)}	=1\,,
\end{align*}
and define for $\da\geq\rhat$ functions $\eta_{\da}\in\ci(\rthree)$ by 
\[\eta_{\da}(x):=\tilde{\eta}(r(x)/\da).\] These functions satisfy
$\supp\eta_{\da}\subset\cU_{\da}$ as well as $\eta_{\da}=1$ on $\cU_{2\da}$ and 
will later be used for particular cut-off procedures. The usual Lebesgue and 
Sobolev spaces will be denoted by $\ltom$, $\Hmom$ and
\begin{align}\label{equ:not_spaces}
	\Rom:=\setb{E\in\ltom}{\rot E\in\ltom}\,,
	\qquad\quad
	\Dom:=\setb{E\in\ltom}{\div E\in\ltom}\,,
\end{align}
where we prefer to write $\rot$ instead of $\operatorname{curl}$. However, for our 
purposes this spaces are not rich enough, as even for square-integrable right hand 
sides the system \eqref{equ:int_max-sys}, \eqref{equ:int_bd-cond} does not admit 
square-integrable solutions (\,cf. \cite{picard_randwertaufgaben_1981}, 
\cite{osterbrink_time-harmonic_2019}\,). 
Hence we have to generalize the solution concept and work in polynomially weighted 
Sobolev spaces. For $\rho:=(1+r^2)^{1/2}$, $m\in\N$, and $t\in\reals$ we introduce
\begin{align*}
	\lttom :=\setb{u\in\ltlocom}{\rho^{t}u\in\ltom}\,,
\end{align*}
as well as
\begin{align*}
  \Hmtom &:=\setb{u\in\lttom}{\p^{\alpha}\hspace*{-0.2mm}u\in\lttom
  			\text{ for all }|\alpha|\leq m}\,,\\
  \hmtom &:=\setb{u\in\lttom}{\p^{\alpha}\hspace*{-0.2mm}u\in\lttpalom
   			\text{ for all }|\alpha|\leq m}\,,\\[-18pt]
\end{align*}
\begin{alignat*}{2}
	\Rtom &:=\setb{E\in\lttom}{\rot E\in\lttom}\,,\qquad\quad 
		  &\rtom &:=\setb{E\in\lttom}{\rot E\in\lttpoom}\,,\\
	\Dtom &:=\setb{H\in\lttom}{\div H\in\lttom}\,,\qquad\quad 
		  &\dtom &:=\setb{H\in\lttom}{\div H\in\lttpoom}\,.
\end{alignat*} 
We do not distinguish between vector fields resp.\;functions and (\,in accordance
with \eqref{equ:not_spaces}\,) we skip the weight if $t=0$, i.e., 
\begin{align*}
	\Hgen{}{1}{}(\Om)=\Hgen{}{1}{0}(\Om)\,,
	\qquad
	\rgen{}{}{}(\Om)=\rgen{}{}{\mathsf{0}}(\Om)\,,
	\qquad 
	\Dgen{}{}{}(\Om)=\Dgen{}{}{\mathsf{0}}(\Om)\,,
	\qquad\ldots\,.
\end{align*}
\noindent
If $\Gat\neq\emptyset$, homogeneous scalar, tangential or normal traces 
are encoded in 
\begin{alignat*}{3}
    \Hoztom &:=\ovl{\dsp\cictom}^{\,\norm{\cdot}_{\Hgen{}{1}{}(\Om)}}\,,\quad &
    \Rztom &:=\ovl{\dsp\cictom}^{\,\norm{\cdot}_{\Rgen{}{}{}(\Om)}}\,, \quad &
    \Dztom &:=\ovl{\dsp \cictom}^{\,\norm{\cdot}_{\Dgen{}{}{}(\Om)}}\,,\\
\intertext{as well as}
    \Hottom &:=\ovl{\dsp\cictom}^{\,\norm{\cdot}_{\Hgen{}{1}{t}(\Om)}}\,,\quad &
    \Rttom &:=\ovl{\dsp\cictom}^{\,\norm{\cdot}_{\Rgen{}{}{t}(\Om)}}\,, \quad &
    \Dttom &:=\ovl{\dsp \cictom}^{\,\norm{\cdot}_{\Dgen{}{}{t}(\Om)}}\,,\notag\\[-6pt]
    &&&&&\\[-6pt]
    \hottom &:=\ovl{\dsp\cictom}^{\,\norm{\cdot}_{\hgen{}{1}{t}(\Om)}}\,,\quad &
    \rttom &:=\ovl{\dsp\cictom}^{\,\norm{\cdot}_{\rgen{}{}{t}(\Om)}}\,, \quad &
    \dttom &:=\ovl{\dsp \cictom}^{\,\norm{\cdot}_{\dgen{}{}{t}(\Om)}}\,.\notag
\end{alignat*}
where the set of test fields (\,resp.\;test functions\,) is given by
\begin{align*}
	\cictom:=\setb{\restr{\varphi}{\Om}}
			   	  {\,\varphi\in\cirthree,
			   	   \;\supp\varphi\text{ compact in }\rthree,\;
			       \dist(\supp \varphi, \mathsf{\Ga}_{1})>0\,}\,.
\end{align*} 
We emphasize that in the case of a bounded domain, weighted and unweighted spaces 
coincide. Moreover, by \cite[Lemma 2.2]{osterbrink_time-harmonic_2019}, see also  
\cite[Theorem 4.5]{bauer_maxwell_2016}, it holds 
\begin{align}\label{equ:prel_weak=strong_1}
   \begin{split}
		\Hottom &=\setb{u\in\Hotom}{\scpltom{u}{\div\Phi}
		         =-\scpltom{\nabla u}{\Phi}\text{ for all }\Phi\in\cicnom}\,,\\
	    \Rttom &=\setb{E\in\Rtom}{
		        \scpltom{E}{\rot \Phi}=\scpltom{\rot E}{\Phi}
		        \text{ for all }\Phi\in\cicnom}\,,\\
	    \Dttom &=\setb{H\in\Dtom}{
		        \scpltom{H}{\nabla\phi}=-\scpltom{\div H}{\phi}
		        \text{ for all }\phi\in\cicnom}\,,
	\end{split}
\end{align}
and
\begin{align}\label{equ:prel_weak=strong_2}
    \begin{split}
	    \hottom &=\setb{u\in\hotom\,}{\scpltom{u}{\div \Phi}
	      		 =-\scpltom{\nabla u}{\Phi}\text{ for all }\Phi\in\cicnom}\,,\\
        \rttom &=\setb{E\in\rtom}{\scpltom{E}{\rot \Phi}
        		 =\scpltom{\rot E}{\Phi}\text{ for all }\Phi\in\cicnom}\,,\\
	    \dttom &=\setb{H\in\dtom}{\scpltom{H}{\nabla \phi}
	    		 =-\scpltom{\div H}{\phi}\text{ for all }\phi\in\cicnom}\,.
    \end{split}	
\end{align}
Equipped with their natural inner products, all these spaces are Hilbert spaces. 
Vanishing rotation resp. divergence will be indicated by an index zero in the 
lower left corner, e.g., 
\begin{alignat*}{2}
	\zrtom &:=\setB{E\in\rtom}{\rot E=0}\,,
	\qquad\qquad &
	\zdtom &:=\setB{E\in\dtom}{\div E=0}\,,\\
	\zrttom &:=\zrtom\cap\rttom\,,
	\qquad\qquad &
	\zdttom &:=\zdtom\cap\dttom\,.
\end{alignat*}
For simplification and to shorten notation we write
\begin{align*}
	\sfV_{<s}:=\bigcap_{t<s}\sfV_{t}
	\qquad\text{and}\qquad
	\sfV_{>s}:=\bigcup_{t>s}\sfV_{t}
	\qquad
	(\,s\in\reals\,)\,,
\end{align*}
for $\V_{\mathrm{t}}$ being any of the spaces above and skip the space reference, 
i.e.,
\begin{align*}
	\Hmt=\Hmtom\,,
	\qquad
	\rtt=\rttom\,,
	\qquad
	\Dt=\Dtom\,,
	\qquad
	\hmtn=\hmtnom\,,
	\qquad\ldots\,,
\end{align*}  
if $\Om=\rthree$. 
\begin{defi}
	Let $\ka\geq0$. We call a transformation $\ga$ 
	``$\ka$-decaying'', if   
	\vspace*{4pt}
	\begin{itemize}[itemsep=6pt]
		\item $\map{\ga}{\Om}{\rttt}$ is an $\li$-matrix field,
		\item $\ga$ is symmetric, i.e.,
			  \begin{align*}
			  		\forall\;E,H\in\ltom:\quad
			  		\scpltom{E}{\ga H}=\scpltom{\ga E}{H}\,,
			  \end{align*}
		\item $\ga$ is uniformly positive definite, i.e.,
			  \begin{align*}
			  	\exists\;c>0\;\;\forall\;E\in\ltom:\quad\scpltom{E}{\ga E}
			  	\geq c\cdot\normltom{E}^2\,,
			  \end{align*}
		\item $\ga$ is asymptotically a multiple of the identity, i.e.,
			  \begin{align*}
			  	\ga=\ga_{0}\cdot\mathbbm{1}+\hat{\ga}\text{ with }\ga_{0}\in
			  	\reals_{+}\text{ and }
			  	\hat{\ga}=\mathcal{O}\big(r^{-\ka}\big)\text{ as }
			  	r\To\infty\,.
			  \end{align*}
	\end{itemize}
\end{defi}
\noindent
\begin{genass}\label{gen-ass}
From now on and through this paper we assume the following:
\begin{itemize}
	\item $\Om\subset\rthree$ is an exterior weak Lipschitz domain with boundary 
		  $\Ga$, decomposed into two weak Lipschitz parts $\Ga_{1}$ and 
		  $\Ga_{2}=\Ga\sm\ovl{\Ga}_{1}$ with weak Lipschitz interface $\ovl{\Ga}_{1}\cap\ovl{\Ga}_{2}$
		  as introduced in the beginning of this section.
	\item There exists $\ka\geq 0$ such that $\eps=\eps_{0}\cdot\mathbbm{1}+\hat{\eps}$ and 
		  $\mu=\mu_{0}\cdot\mathbbm{1}+\hat{\mu}$ are $\ka$-decaying. 
\end{itemize}
\end{genass}
\noindent
For most of our results we need the slightly stronger assumption on the 
perturbations $\hat{\eps}$ and $\hat{\mu}$. That is, $\hat{\eps}$ 
resp.\;$\hat{\mu}$ have to be differentiable outside of an arbitrarily large 
ball with decaying derivative. More precisely: 
\begin{defi} 
	Let $\ka\geq0$. We call a transformation $\ga$ 
	\emph{``$\ka-\co-$decaying''}, if 
	\vspace*{4pt} 
	\begin{itemize}[itemsep=6pt]
		\item $\ga=\ga_{0}\cdot\mathbbm{1}+\hat{\ga}$ is $\ka$-decaying,
		\item and for some $\rtil>\rhat$ we have
			  \begin{align*}
			  	\hat{\ga}\in\co(\cU_{\rtil})
			  	\quad\text{with}\quad
			  	\p_{j}\hat{\ga}=\calO\big(r^{-1-\ka}\big)
			  	\;\text{ as }\;
		  		r\To\infty\,,\qquad (\,j=1,2,3\,).
			  \end{align*}
	\end{itemize}
\end{defi}
\noindent
Note that a $\ka$-decaying (\,resp.\;$\ka-\co-\,$decaying\,) transformation 
$\ga$ is pointwise invertible for sufficiently large $x$. In this sense, 
$\ga^{-1}$ is $\ka$-decaying (\,resp. $\ka-\co-\,$decaying) as well. 
Moreover,\\[-11pt] 
\begin{align*}
	\scp{\,\cdot\,}{\,\cdot\,}_{\ltgaom}	:=\scpltom{\,\ga\,\cdot\,}{\,\cdot\,}
	\quad\;\;\;\;\text{resp.}\quad\quad\;
	\scp{\,\cdot\,}{\,\cdot\,}_{\lttga(\Om)}
		:=\scpltom{\,\ga\rho^{t}\,\cdot\,}{\,\rho^{t}\,\cdot\,}
\end{align*}
define inner products on $\ltom$\;resp.\;$\lttom$ inducing norms equivalent to 
the standard ones. Thus 
\begin{align*}
	\ltgaom:=\big(\,\ltom,\scp{\,\cdot\,}{\,\cdot\,}_{\ltgaom}\,\big) 
	\qquad\;\text{and}\qquad\;
	\lttga(\Om):=\big(\,\lttom,\scp{\,\cdot\,}{\,\cdot\,}_{\lttga(\Om)}\,\big)	
\end{align*}
are Hilbert spaces and we use $\dsp\oplus_{\ga}$, $\dsp\oplus_{\mathrm{t},\ga}$ 
resp.\;$\dsp\perp_{\ga}$, $\dsp\perp_{\mathrm{t},\ga}$ to indicate orthogonal sum 
and orthogonal complement in this spaces. If $\ga=\mathbbm{1}$ we put 
$\oplus_{\ga}=:\oplus$ as well as $\perp_{\ga}\,=:\,\perp$. Finally we introduce 
for $s\in\reals$ the (\,weighted\,) 
\emph{``Dirichlet-Neumann fields''}
\begin{align*}
	\gharmstnom:=\zrstom\cap\ga^{-1}\zdsnom
	\,,\qquad\qquad
	\harmgen{}{s}{\Gat}{\Gan}(\Om):=\harmgen{\mathbbm{1}}{s}{\Gat}{\Gan}(\Om)\,,
\end{align*}
where as before we skip the weight if $s=0$.  
%
%
\section{The Static Problem $\om=0$}
\label{sec:e-m-static} 
%
%
We start our considerations with the supposedly simpler case of 
electro-magneto-statics, which in fact possesses its own difficulties. First, as 
$\Om$ is an exterior domain we are forced to work in polynomially weighted Sobolev 
spaces. Second, for $\om=0$ the time-harmonic Maxwell 
system \eqref{equ:int_max-sys},\eqref{equ:int_bd-cond}, i.e.,
\begin{alignat*}{4}
	\rot E&=G\;\,&&\;\text{in}\;\;\,\Om\,,
	\qquad\qquad & 
	n\times E &=0\;&&\;\text{on}\;\;\,\Gat\,,\\[2pt]
	\rot H&=F\;\;&&\;\text{in}\;\;\,\Om\,,
	\qquad\qquad & 
	n\times H&=0\;\,&&\;\text{on}\;\;\,\Gan\,,
\end{alignat*}
is no longer coupled and in order to determine $E$ and $H$ we have to add two 
more equations\footnote{For $\om\neq 0$ these equations are implicitly given, as 
by differentiating \eqref{equ:int_max-sys} we immediately get
\begin{align*}
	i\om\div\eps E=\div(-\rot H+i\om\eps E)=-\div F\,,
	\quad\quad
	i\om\div\mu H =\div(\rot E+i\om\mu H)=\div G
	\quad\quad
	\text{in }\Om\,.
\end{align*}}
\begin{align*}
	\div\eps E=f\,,
	\qquad\qquad
	\div\mu H=g\,,
	\qquad\qquad
	\text{in }\Om\,,
\end{align*} 
as well as additional boundary conditions 
\begin{align*}
	n\cdot\eps E=0\;\;\,\text{on}\;\;\,\Gan
	\,,\qquad\qquad 
	n\cdot\mu H=0\;\,\,\text{on}\;\;\,\Gat\,.
\end{align*}
The resulting boundary value problems of electro- resp.\;magneto-statics 
(\,cf. \eqref{equ:int_static-system}\,)
\begin{alignat}{4}
	\rot E&=G
	\,,\qquad\qquad &
	\div\eps E&=f
	\,,\qquad\qquad &
	n\times E &=0\;\;\,\text{on}\;\;\Gat
	\,,\qquad\qquad &
	n\cdot\eps E &=0\;\;\,\text{on}\;\;\Gan\,,\label{equ:stp_problem_es}\\[2pt]
	\rot H&=F
	\,,\qquad\qquad &
	\div\mu H&=g
	\,,\qquad\qquad &
	n\times H &=0\;\;\,\text{on}\;\;\Gan
	\,,\qquad\qquad &
	n\cdot\mu H &=0\;\;\,\text{on}\;\;\Gat\,,\label{equ:stp_problem_ms}
\end{alignat}
still have non-trivial but finite-dimensional kernels $\vHarmtn{\eps}(\Om)$ and 
$\vHarmnt{\mu}(\Om)$, respectively, demanding for finitely many orthogonality 
constraints to achieve unique solutions. Due to the similarity between 
\eqref{equ:stp_problem_es} and \eqref{equ:stp_problem_ms} we concentrate on the 
electro-static problem \eqref{equ:stp_problem_es}, keeping in mind that 
interchanging $\Gat$ and $\Gan$ as well as $\eps$ and $\mu$ we also solve the 
magneto-static system.\\[12pt]
Let $\Theta$ be a domain in $\rthree$. Considering the densely defined and 
closed linear operators 
\begin{align*}
	\maps{\A_{1}:=\grad_{\Gat}}
		 {\dod(\A_{1})&:=\Hozt(\Theta)\subset\lt(\Theta)}
		 {\lteps(\Theta)}{w}{\nabla w}\,,\\
	\maps{\A_{2}:=\rot_{\Gat}}
		 {\dod(\A_{2})&:=\Rzt(\Theta)\subset\lteps(\Theta)}
		 {\lt(\Theta)}{u}{\rot u}\,,
\end{align*}
the Hilbert space adjoints are
(\,cf. \cite[Lemma 2.2]{osterbrink_time-harmonic_2019}, 
\cite[Theorem 4.5]{bauer_maxwell_2016}\,)
\begin{align*}
	\maps{\adj{\A}_{1}=\adj{\grad}_{\Gat}=-\div_{\Gan}\eps}
		 {\dod(\adj{\A}_{1})=\eps^{-1}\Dzn(\Theta) &\subset\lteps(\Theta)}
		 {\lt(\Theta)}{u}{-\div\eps u}\,,\\
	\maps{\adj{\A}_{2}=\adj{\rot}_{\Gat}=\eps^{-1}\rot_{\Gan}}
	     {\dod(\adj{\A}_{2})=\Rzn(\Theta) &\subset\lt(\Theta)}
	     {\lteps(\Theta)}{u}{\eps^{-1}\rot u}\,.
\end{align*}
These operators satisfy 
\begin{align}\label{equ:stp_complex}
	\rg(\A_{1})\subset \ker(\A_{2})\,,
	\qquad\quad\qquad\quad
	\rg(\adj{\A}_{2})\subset\ker(\adj{\A}_{1})\,,	
\end{align}
and by the projection theorem the Helmholtz-type decompositions
\begin{align}\label{equ:stp_decomp}
	\left\{\;\;
		\begin{aligned}
			\lt(\Theta) 
			&=\ovl{\rg(\A_{1})}
	          \oplus_{\eps}
	   		  \ker(\overset{}{\A}^{\hspace*{-0.05cm}\ast}_{1})
	   		\quad\quad & 
			\lt(\Theta) 
			&=\ovl{\rg(\overset{}{\A}^{\hspace*{-0.05cm}\ast}_{2})}
	   	  	  \oplus_{\eps}\ker(\A_{2})\,,\\
	   		&=\ovl{}\ovl{\nabla\Hozt(\Theta)}
	   		  \oplus_{\eps}\eps^{-1}\zdzn(\Theta)\,,
	   		\quad\quad & 
		    &=\eps^{-1}\ovl{\rot\Rzn(\Theta)}
		      \oplus_{\eps}\zrzt(\Theta)\,,	\\[4pt]
		    \zrzt(\Theta) 
		    &=\ker(\A_{2})
			\quad\quad &
			\eps^{-1}\zdzn(\Theta) 
			&=\ker(\overset{}{\A}^{\hspace*{-0.05cm}\ast}_{1})\\
			&=\ovl{\rg(\A_{1})}
			  \oplus_{\eps}
			  \big(\,\ker(\overset{}{\A}^{\hspace*{-0.05cm}\ast}_{1})
			  \cap\ker(\A_{2})\,\big)
			\quad\quad &
			&=\ovl{\rg(\overset{}{\A}^{\hspace*{-0.05cm}\ast}_{2})}
			  \oplus_{\eps}\big(\,\ker(\A_{2})
			  \cap\ker(\overset{}{\A}^{\hspace*{-0.05cm}\ast}_{1})\,\big)\\
			&=\ovl{\nabla\Hozt(\Theta)}
			  \oplus_{\eps}\vHarmtn{\eps}(\Theta)\,,
			\quad\quad &
			&=\eps^{-1}\ovl{\rot\Rzn(\Theta)}
			  \oplus_{\eps}\vHarmtn{\eps}(\Theta)
		\end{aligned}
	\right.
\end{align}
hold true. As shown in \cite{pauly_solution_2019}, rewriting 
\eqref{equ:stp_problem_es} into
\begin{align}
	\A_{2} E=G\,,
	\qquad
	\overset{}{\A}^{\hspace*{-0.05cm}\ast}_{1}E=f\,,
	\qquad 
	E\in\dod(\A_{2})\cap\dod(\overset{}{\A}^{\hspace*{-0.05cm}\ast}_{1})
\end{align}
and using \eqref{equ:stp_decomp} together with standard functional analysis tools, 
we immediately obtain an $\lt$-solution theory for electro-magneto-statics, 
provided $\Theta$ satisfies \emph{``Weck's selection theorem''}, a compactness 
result comparable to Rellich's selection theorem well suited for Maxwell's 
equations.
\begin{defi}\label{def:stp_WST}
    A domain $\Theta\subset\rthree$ satisfies 
    {"Weck's selection theorem"} (WST) 
    (\,or possesses the {"Maxwell compactness property"}\,) if the embedding
    \begin{align}\label{equ:stp_WST}
	 	\Rzt(\Theta)\cap\eps^{-1}\Dzn(\Theta)
	 	\xhookrightarrow{\hspace{0.4cm}}\lt(\Theta)\text{ is compact\,.}
    \end{align}
\end{defi}
\noindent
In particular, as shown in \cite{bauer_maxwell_2016}
(see also \cite[Section 5]{pauly_solution_2019}), it holds:
\begin{lem}\label{lem:stp_fa-toolbox}
	Let $\Theta\subset\rthree$ be a bounded weak Lipschitz domain with boundary 
	$\Ga$ and weak Lipschitz interfaces $\Gat$ and $\Gan:=\Ga\sm\ovl{\Ga}_{1}$. 
	Then Weck's selection theorem holds true and implies the following:
	\begin{enumerate}[leftmargin=0.9cm,label=$(\roman*)$]
		\item $\mathrm{(}$\,\emph{Maxwell estimate}\,$\mathrm{)}$ There is 
			  $c>0$ such that for all
			  $E\in\Rgen{}{}{\Gat}(\Theta)\cap\eps^{-1}\Dzn(\Theta)
			  \cap\vHarmv{\eps}{\Gat}{\Gan}(\Theta)^{\perp_{\eps}}$ 
			  \begin{align*}
				\norm{E}_{\lt(\Theta)}
					\leq c\,\Big(\norm{\rot E}_{\lt(\Theta)}
			     		 +\norm{\div\eps E}_{\lt(\Theta)}\Big)\,.	  	
			  \end{align*}
		\item $\mathrm{(}$\,\emph{Finite dimensional kernel}\;$\mathrm{)}$ The 
			  unit ball in $\vHarmtn{\eps}(\Theta)$ is compact, i.e., 
			  \[\dim\vHarmtn{\eps}(\Theta)<\infty\,.\]
		\item $\mathrm{(}$\,\emph{Closed ranges}\,$\mathrm{)}$ The ranges of 
			  ${\operatorname{grad}_{\Gat}}$and\,
			  ${\rot_{\Gat}}$\,resp.\;$\div_{\Gan}\eps$ and 
			  ${\eps^{-1}\rot_{\Gan}}$\,are closed, i.e., 
			  \begin{alignat*}{2}
				 \ovl{\nabla\Hgen{}{1}{\Gat}(\Theta)}
				   &=\nabla\Hgen{}{1}{\Gat}(\Theta)\,,
				 \qquad\qquad &
				 \ovl{\rot\Rgen{}{}{\Gat}(\Theta)}
				   &=\rot\Rgen{}{}{\Gat}(\Theta)\,,\\
				 \ovl{\div\Dzn(\Theta)}
				   &=\div\Dzn(\Theta)\,,
				 \qquad\qquad &
				 \ovl{\rot\Rgen{}{}{\Gan}(\Theta)}
				   &=\rot\Rgen{}{}{\Gan}(\Theta)\,,
			  \end{alignat*}
			  and the following Helmholtz type decompositions are valid
			  \begin{alignat*}{2}
				\lt(\Theta) &=\nabla\Hozt(\Theta)
							  \oplus_{\eps}
							  \eps^{-1}\zdzn(\Theta)\,,
				\qquad\quad & 
				\lt(\Theta) &=\eps^{-1}\rot\Rzn(\Theta)
							  \oplus_{\eps}
							  \zrzt(\Theta)\,,\\
				\zrzt(\Theta) &=\nabla\Hozt(\Theta)
								\oplus_{\eps}
								\vHarmtn{\eps}(\Theta)\,,
				\qquad\quad &
				\eps^{-1}\zdzn(\Theta) &=\eps^{-1}\rot\Rzn(\Theta)
				  						 \oplus_{\eps}
				  						 \vHarmtn{\eps}(\Theta)\,.
			  \end{alignat*}
	\end{enumerate}
\end{lem}
\begin{rem}
	In the latter lemma and the previous arguments (\,involving $\Theta$\,) it is 
	sufficient that $\eps$ is $\ka-$decaying with $\ka\geq0$. 
\end{rem}
\noindent
While Weck's selection theorem holds true for bounded weak Lipschitz 
domains, it fails for unbounded such as exterior domains 
(\,cf.\;\cite{bauer_maxwell_2016}, \cite{jochmann_compactness_1997} and also
\cite{fernandes_magnetostatic_1997} for strong Lipschitz domains). Thus, we cannot 
retreat on the functional analysis toolbox from \cite{pauly_solution_2019}, 
especially Lemma \ref{lem:stp_fa-toolbox}, to solve system 
\eqref{equ:stp_problem_es}. Instead we will use a slightly weaker version of 
\eqref{equ:stp_WST} to prove similar results in weighted $\lt-\,$spaces. 
More precisely, as $\Omda$ is a bounded weak Lipschitz domain with boundary 
$\Ga=\ovl{\Ga}_{\da}\cup\ovl{\Ga}_{2}$, Lemma \ref{lem:stp_fa-toolbox} yields, 
e.g.,
\begin{align*}
	\forall\;\da\geq\rhat:        
    \quad
    \Rgen{}{}{\Ga_{\da}}(\Omda)\cap\eps^{-1}&\Dzn(\Omda)
    \xhookrightarrow{\hspace{0.4cm}}\lt(\Omda)\quad\text{ is compact\,.}
\end{align*}
Hence by \cite[Lemma 3.3]{osterbrink_time-harmonic_2019} it holds:
\begin{theo}[Weck's local selection theorem]\label{thm:stp_WLST}
	The embedding 
	\[\Rztom\cap\eps^{-1}\Dznom\xhookrightarrow{\hspace{0.4cm}}\ltlocomb\] 
    is compact. Equivalently for all $s,t\in\reals$ with $t<s$ the embedding 
    \[\Rstom\cap\eps^{-1}\Dsnom\xhookrightarrow{\hspace{0.4cm}}\lttom\] 
    is compact.
\end{theo}
\noindent
As we will show in the following, by Theorem \ref{thm:stp_WLST} we are indeed able 
to reconstruct the results of Lemma \ref{lem:stp_fa-toolbox} in the framework of 
weighted Sobolev spaces, leading to a solution theory for 
\eqref{equ:stp_problem_es}\;resp.\;\eqref{equ:stp_problem_ms} in exterior domains.
%
\subsection{Poincar\'e and Maxwell Estimates in Exterior Domains} %
%
We start out proving a weighted version of the Poincar\'e estimate. 
From \cite[Lemma 15]{pauly_functional_2009}, see 
also \cite[Poincare's estimate III]{leis_initial_2013}, we have 
\begin{align}\label{equ:stp_poincare-smooth}
	\norm{\phi}_{\ltmo(\Om)}
		\leq\|\,r^{-1}\phi\,\|_{\ltom}
		\leq\;2\,\normltom{\nabla \phi}
	\qquad
	\forall\;\phi\in\cicgom\,,
\end{align}
which by continuity extends to all $u\in\homog(\Om)$ and can even be generalized 
to functions in $\homo(\Om)$.
\begin{lem}\label{lem:stp_poincare}
	The following Poincar\'e 
	estimate holds:
	\begin{align*}
		\exists\;c>0\quad\forall\;u\in\homoom\,:
		\quad\qquad
		\norm{u}_{\ltmoom}\leq c\,\norm{\nabla u}_{\ltom}
	\end{align*}
\end{lem}
\begin{proof}
	For $u\in\homoom$ it holds $\eta_{\rhat}\hspace*{0.02cm}u\in\homog(\Om)$ 
	and from \eqref{equ:stp_poincare-smooth} we obtain
	\begin{align}\label{equ:stp_poincare-mud}
		\norm{u}_{\ltmoom}
			&\leq 2\,\norm{\nabla(\eta_{\rhat}\hspace*{0.02cm}u)}_{\ltmoom}
			 +\norm{(1-\eta_{\rhat})\hspace*{0.02cm}u}_{\ltmoom}
			 \leq c\,\Big(\norm{\nabla u}_{\ltom}
			 +\norm{u}_{\lt(\Om_{2\rhat})}\Big)\,,
	\end{align}
	with $c>0$. Assuming the asserted estimate is wrong, there exists a sequence 
	$(u_{n})_{n\in\N}\subset\homo(\Om)$ with
	\begin{align*}
		\norm{u_{n}}_{\ltmo(\Om)}=1
		\qquad\text{and}\qquad
		\norm{\nabla {u_n}}_{\ltom}<\frac{1}{n}
		\xrightarrow{\;n\to\infty\;} 0\,.
	\end{align*}
	Hence, $(u_{n})_{n\in\N}$ is bounded in $\ho(\Om_{2\rhat})$ and by Rellich's 
	selection theorem\footnote{Note that also Rellich's selection theorem 
	holds in bounded weak Lipschitz domains 
	(\,cf. \cite[Theorem 4.8]{bauer_maxwell_2016}\,).} we can extract a 
	subsequence $(u_{\pi(n)})_{n\in\N}$ converging in $\lt(\Om_{2\rhat})$. By 
	\eqref{equ:stp_poincare-mud} the sequence $(u_{\pi(n)})_{n\in\N}$ is even a 
	Cauchy sequence in $\homo(\Om)$ and therefore converging to some 
	$u\in\homo(\Om)$ with $\nabla u=0$. Consequently $u$ is constant in $\Om$ and 
	as $u\in\ltmoom$ we have $u=0$, a contradiction.
\end{proof}
\noindent
Similarly, Weck's local selection theorem yields a weighted version of the Maxwell 
estimate. Again we start with testfields $\Phi\in\cicgom$ stating that by 
\eqref{equ:stp_poincare-smooth} and 
$-\Delta\Phi=\rot\rot\Phi-\nabla\div\Phi$ we have 
\begin{align}\label{equ:stp_div-rot-est_smooth}
	\normltmoom{\Phi}
	    \leq c\,\normltom{\nabla\Phi}
		\leq c\,\left(\,\normltom{\rot\Phi}
		     +\normltom{\div\Phi}\,\right)
	\qquad
	\forall\;\Phi\in\cicgom\,.
\end{align}
which directly extends to $\Phi\in\homog(\Om)$ and can also be 
generalized. 
\begin{lem}\label{lem:stp_reg-res}
		Let $s\in\reals$, $\rtil>\rhat$, and 
		$\Xi\subset\cU_{\rtil}\subset\rthree$ be an 
		exterior domain with $\dist(\Xi,\Sp_{\rtil})>0$. Furthermore, let 
		$\eps$ be $\ka-\co-$decaying with 
		$\ka>0$ such that $\eps\in\co(\cU_{\rtil})$. Then the conditions 
		$E\in\lts(\cU_{\rtil})$, $\rot E\in\ltspo(\cU_{\rtil})$, and 
		$\div\eps E\in\ltspo(\cU_{\rtil})$ imply $E\in\hos(\Xi)$ and it holds
		\begin{align*}
			\norm{E}_{\hos(\Xi)}
				\leq c\,\Big(\norm{E}_{\lts(\cU_{\rtil})}
							 +\norm{\rot E}_{\ltspo(\cU_{\rtil})}
							 +\norm{\div\eps E}_{\ltspo(\cU_{\rtil})}\Big)
		\end{align*}
		with $c>0$ independent of $E$.
\end{lem}
\begin{proof}
	This regularity result is a direct consequence of 
	\cite[Lemma 4.2]{kuhn_regularity_2010}\,.
	A detailed proof can be found in 
	\cite[Korollar 3.7]{pauly_niederfrequenzasymptotik_2003}.
\end{proof}
\begin{rem}
	By obvious modifications on $\eps$ and the assumptions imposed on $E$, the 
	previous result can also be formulated for the bold Hilbert spaces, e.g., 
	$\Hos(\Xi)$. 	
	Beyond that it may even be generalized to higher regularity for $E$, e.g., 
	$E\in\hgen{}{m}{s}(\Xi)$. We also note that the assumptions on $\eps$ may be 
	reduced to a $\ka-$decaying $\eps=\eps_{0}\cdot\mathbbm{1}+\hat{\eps}$ with 
	$\ka>0$ and $\hat{\eps}\in\co(\cU_{\rtil})$ such that 
	$\p_{j}\hat{\eps}=\calO\big(r^{-1}\big)$.
\end{rem}
\begin{lem}\label{lem:stp_max-est_mud}
	Let $\eps$ be $\ka-\co-$decaying with order $\ka>0$. Then there exist 
	$c,\da>0$ such that 
	\begin{align*}
		\norm{E}_{\ltmoom}
			\leq c\,\Big(\normltom{\rot E}
			     +\normltom{\div\eps E}
			     +\norm{E}_{\lt(\Om_{\da})}\Big)\,
	\end{align*}
	holds for all $E\in\rmoom\cap\eps^{-1}\dmoom$.
\end{lem}
\begin{proof}
By assumption $\eps$ is of the form $\eps=\eps_{0}\cdot\mathbbm{1}+\hat{\eps}$ 
with $\eps_{0}\in\reals_{+}$ and there exists $\rtil>\rhat$ such that
\begin{align}\label{equ:stp_trans-ass}
	\hat{\eps}\in\co(\cU_{\rtil})
	\quad\text{with}\quad
	\hat{\eps}=\calO\big(r^{-\ka}\big)\,,\;\,	
	\p_{j}\hat{\eps}=\calO\big(r^{-1-\ka}\big)
	\,\;\text{ as }\;\,
	r\To\infty\,,\qquad (\,j=1,2,3\,)\,.
\end{align}
Using the cut-off function from above, we define $\tilde{E}:=\eta_{\rtil} E$ and 
as Lemma \ref{lem:stp_reg-res} yields $E\in\homo(\supp\eta_{\rtil})$, we have 
$\tilde{E}\in\homog(\Om)$. Hence by \eqref{equ:stp_div-rot-est_smooth}
\begin{align*}
		\|\,\tilde{E}\,\|_{\ltmoom}
		&\leq c\,\Big(\|\rot\tilde{E}\,\|_{\lt(\cU_{\rtil}) }
		     +\|\div\eps_{0}\tilde{E}\,\|_{\lt(\supp\eta_{\rtil})}\Big)\\
		&\leq c\,\Big(\|\rot\tilde{E}\,\|_{\lt(\cU_{\rtil})}
		     +\|\div\eps\tilde{E}\,\|_{\lt(\supp\eta_{\rtil})}
		     +\|\div\hat{\eps}\tilde{E}\,\|_{\lt(\supp\eta_{\rtil})}\Big)\\ 
		&\leq c\,\Big(\|\rot\tilde{E}\,\|_{\lt(\cU_{\rtil})}
		     +\|\div\eps\tilde{E}\,\|_{\lt(\cU_{\rtil})}
		     +\sum_{j=1,2,3}\|\hspace*{0.02cm}
		      (\partial_{j}\hat{\eps})\tilde{E}\,\|_{\lt(\supp\eta_{\rtil})}
		     +\|\,\hat{\eps}:\nabla\tilde{E}\,\|_{\lt(\supp\eta_{\rtil})}\Big)\,.      
\end{align*}
With \eqref{equ:stp_trans-ass} and the regularity estimate from 
Lemma \ref{lem:stp_reg-res} we obtain
\begin{align*}
	\|\,\tilde{E}\,\|_{\ltmoom}
		&\leq c\,\Big(\|\rot\tilde{E}\,\|_{\lt(\cU_{\rtil})}
		      +\|\div\eps\tilde{E}\,\|_{\lt(\cU_{\rtil})}
		      +\|\,\tilde{E}\,\|_{\homomka(\supp\eta_{\rtil})}\Big)\\
		&\leq c\,\Big(\|\rot\tilde{E}\,\|_{\lt(\cU_{\rtil})}
		      +\|\div\eps\tilde{E}\,\|_{\lt(\cU_{\rtil})}
		      +\|\,\tilde{E}\,\|_{\ltmomka(\cU_{\rtil})}\Big)\\
		&\leq c\,\Big(\|\rot\tilde{E}\,\|_{\ltom}
		      +\|\div\eps\tilde{E}\,\|_{\ltom}
		      +\|\,\tilde{E}\,\|_{\ltmomkaom}\Big)\,, 	
\end{align*}
such that by
\begin{align*} 
	\rot\tilde{E}=\eta_{\rtil}\rot E+(\nabla\eta_{\rtil})\times E
	\,,\qquad\qquad
	\div\eps\tilde{E}=\eta_{\rtil}\div\eps E+(\nabla\eta_{\rtil})\cdot\eps E\,,
\end{align*}
we end up with
\begin{align*}
	\normltmoom{E}
		&\leq c\,\Big(\|\,\tilde{E}\,\|_{\ltmoom}
		+\|\,E\,\|_{\lt(\Om_{2\rtil})}\Big)\\
		&\leq c\,\Big(\|\rot\tilde{E}\,\|_{\ltom}
		      +\|\div\eps\tilde{E}\,\|_{\ltom}
		      +\|\,\tilde{E}\,\|_{\ltmomkaom}
		      +\|\,E\,\|_{\lt(\Om_{2\rtil})}\Big)\\
		&\leq c\,\Big(\|\rot E\,\|_{\ltom}
		      +\|\div\eps E\,\|_{\ltom}
		      +\|\,E\,\|_{\ltmomkaom}\Big)\,.
\end{align*}
Finally, as $\ka>0$, the assertion follows by 
\begin{align*}
	\|\,E\,\|_{\ltmomkaom}^2
		&=\norm{E}_{\ltmomka(\Om_{\da})}^2
		  +\norm{E}_{\ltmomka(\cU_{\da})}^2
		 \leq\norm{E}_{\lt(\Om_{\da})}^2
		  +(1+\da^{2})^{-\ka}
		  \cdot\norm{E}_{\ltmoom}^2
\end{align*}
choosing $\da>\rhat$ big enough.
\end{proof}
Now, analogously to the proof of Lemma \ref{lem:stp_poincare}, we use 
Theorem \ref{thm:stp_WLST} to eliminate the extra term on the right hand side. But 
unlike there, here the kernels of the involved operators 
``\hspace*{0.02cm}$\rot$\hspace*{0.02cm}'' 
and ``\hspace*{0.02cm}$\div\eps$\hspace*{0.02cm}'' are not necessarily trivial. 
Therefore, we aim for a weighted version of the Maxwell 
estimate excluding the \emph{Dirichlet-Neumann fields}
\begin{align*}
	\vharmmotn{\eps}(\Om)=\zrmotom\cap\eps^{-1}\zdmonom\,.
\end{align*}
Fortunately, the space $\vharmmotn{\eps}(\Om)$ is only finite dimensional.
\begin{lem}\label{lem:stp_toolbox-reproduced}
	Let $\eps$ be 
	$\ka-\co-\,$decaying with $\ka>0$. Then: 
	\begin{enumerate}[leftmargin=0.9cm,label=$(\roman*)$] 
		\item $\mathrm{(}$\,\emph{Maxwell estimate}\,$\mathrm{)}$ There is $c>0$ 
			  s.\,t. for all 
			  $E\in\rmotom\cap\eps^{-1}\dmonom
			    \cap\vharmmotn{\eps}(\Om)^{\perp_{-1,\eps}}$
			  \begin{align*}
				\normltmoom{E}
					\leq c\cdot\Big(\normltom{\rot E}
					      +\normltom{\div\eps E}\Big)\,.\\[-10pt]	
			  \end{align*}
		\item $\mathrm{(}$\,\emph{Finite dimensional kernel}\,$\mathrm{)}$
			  The unit ball in $\vharmmotn{\eps}(\Om)$ is compact, i.e., 
			  \[\dim\vharmmotn{\eps}(\Om)<\infty\,.\]
		\item $\mathrm{(}$\,\emph{Closed ranges}\,$\mathrm{)}$ The ranges of 
			  $\operatorname{grad}_{\Gat}$, 
			  ${\rot_{\Gat}}${\;and\;}\;$\div_{\Gan}\eps$
			  are \textbf{not} closed, but it holds
			  \begin{align*}
				\;(a)\;\; &\ovl{\nabla\Hoztom}
		 		    		   =\ovl{\nabla\homotom}=\nabla\homotom\,,\\[6pt]
				\;(b)\;\; &\ovl{\rot\Rztom}=\ovl{\rot\rmotom}\\
				   &\phantom{\ovl{\rot\Rztom}}=\rot\rmotom
				    =\rot\Big(\,\rmotom\cap\eps^{-1}\zdmonom
				     \cap\vharmmotn{\eps}(\Om)^{\perp_{-1,\eps}}\Big)\,,\\[6pt]
				\;(c)\;\; &\ovl{\div\Dznom}=\ovl{\div\dmonom}\\
				   &\phantom{\ovl{\div\Dznom}}=\div\dmonom
				    =\div\Big(\dmonom\cap\eps\,\zrmotom
				     \cap\eps\,\vharmmotn{\eps}(\Om)^{\perp_{-1,\eps}}\Big)\,.
			  \end{align*}
	\end{enumerate}
\end{lem}
\begin{proof}
	Statement $(ii)$ just follows by Weck's local selection theorem 
	and Lemma \ref{lem:stp_max-est_mud}. For $(i)$ suppose the estimate 
	is wrong, i.e., there exists $(E_{n})_{n\in\N}\subset\rmotom\cap
	\eps^{-1}\dmonom\cap\vharmmotn{\eps}(\Om)^{\perp_{-1,\eps}}$ with\\[-8pt]
	\begin{align}\label{equ:stp_max-est_ass}
		\normltmoom{E_{n}}=1
		\qquad\text{and}\qquad
		\normltom{\rot E_{n}}+
		\normltom{\div\eps E_{n}}
		\xrightarrow{\;n\to\infty\;} 0\,.
	\end{align}
	Then the sequence $(E_{n})_{n\in\N}$ is bounded in 
	$\rmotom\cap\eps^{-1}\dmonom$ and Weck's local selection theorem provides a 
	subsequence $(E_{\pi(n)})_{n\in\N}$ converging in $\ltlocomb$. By 
	Lemma \ref{lem:stp_max-est_mud} the sequence 
	$(E_{\pi(n)})_{n\in\N}$ is an $\ltmo-$Cauchy-sequence and 
	we obtain 
	\begin{align*}
		E:=\lim_{n\to\infty}E_{\pi(n)}\in\rmotom\cap\eps^{-1}\dmonom
		\qquad\text{with}\qquad	
		\rot E=0
		\qquad\text{resp.}\;\quad
		\div\eps E=0\,.	
	\end{align*}
	Additionally, 
	$(E_{\pi(n)})_{n\in\N}\subset(E_{n})_{n\in\N}\subset\vharmmotn{\eps}
	(\Om)^{\perp_{-1,\eps}}$ such that
	\begin{align*}
		\forall\;H\in\vharmmotn{\eps}(\Om):
		\qquad
		\scp{E}{H}_{\ltmoeps(\Om)}
		 	=\lim_{n\to\infty}\scp{E_{\pi(n)}}{H}_{\ltmoeps(\Om)}=0\,.
	\end{align*}
	hence 
	\begin{align*}
		E\in\vharmmotn{\eps}(\Om)
		    \cap
		    \vharmmotn{\eps}(\Om)^{\perp_{-1,\eps}}=\{0\}\,,	
	\end{align*}
	a contradiction. Let us finally turn to statement $(iii)$. By 
	definition we clearly have\\[-6pt]
	\begin{align}\label{equ:stp_closures}
			\hspace*{-0.5cm}\ovl{\nabla\Hoztom}=\ovl{\nabla\homotom}\,,\qquad
			\ovl{\rot\Rztom}=\ovl{\rot\rmotom}\,,\qquad
			\ovl{\div\Dznom}=\ovl{\div\dmonom}\,.\\[-12pt]\notag
	\end{align}
	Now, for $u^{\nabla}\in \ovl{\nabla\homotom}$ there exists 
	$(u_{n})_{n\in\N}\subset\homotom$ with 
	$\nabla u_{n}\xrightarrow{n\to\infty} u^{\nabla}$ in 
	$\ltom$. The Poincar\'e estimate, Lemma \ref{lem:stp_poincare}, shows 
	that $(u_{n})_{n\in\N}$ is converging in $\ltmoom$ to some $u\in\ltmoom$ 
	and we have
	\begin{align*}
		\scp{u}{\div\Phi}_{\ltom}
			=\lim_{n\to\infty}\scp{u_{n}}{\div\Phi}_{\ltom}
			=-\lim_{n\to\infty}\scp{\nabla u_{n}}{\Phi}_{\ltom}
			=-\scp{u^{\nabla}\hspace*{-0.05cm}}{\Phi}_{\ltom}
		\quad\forall\;\Phi\in\cicnom\,.	
	\end{align*}
	Thus, by \eqref{equ:prel_weak=strong_2} we have $u\in\homotom$ 
	and $\nabla u=u^{\nabla}$, which shows $(a)$. For $(b)$ let 
	$E^{\rot}\in\ovl{\rot\rmotom}$ and  
	$(E_{n})_{n\in\N}\subset\Rztom$ a sequence with 
	$\rot E_{n}\xrightarrow{n\to\infty} E^{\rot}$ in $\ltom$. 
	Using the decompositions from \eqref{equ:stp_decomp} and statement $(a)$, we 
	obtain $E_{n}=E^{\nabla}_{n}+\hat{E}_{n}
	\in\nabla\homotom\oplus_{\eps}\eps^{-1}\,\zdznom$, hence 
	\begin{align*}
		\hat{E}_{n}=E_{n}-E^{\nabla}_{n}
		\in\Rztom\cap\eps^{-1}\zdznom
		\quad\text{with}\quad
		\rot \hat{E}_{n}=\rot E_{n}\xrightarrow{\,\ltom\,} E^{\rot}\,.
	\end{align*}
	As $\hat{E}_{n}\subset\ltom\subset\ltmoom$ and 
	$\vharmmotn{\eps}(\Om)\subset\ltmoom$ is finite-dimensional, we 
	continue splitting
	\begin{align*}
		\hat{E}_{n}=E^{\mathscr{H}}_{n}+\tilde{E}_{n}\hspace*{-0.1cm}
				   \in\vharmmotn{\eps}(\Om)
				   \oplus_{-1,\eps}
				   \vharmmotn{\eps}(\Om)^{\perp_{-1,\eps}}\,,	
	\end{align*}
	and end up with 
	\begin{align*}
		\tilde{E}_{n}=\hat{E}_{n}-E^{\mathscr{H}}_{n}
		\in\rmotom
		\cap\eps^{-1}\zdmonom
		\cap\vharmmotn{\eps}(\Om)^{\perp_{-1,\eps}}
		\,,\quad\quad	
		\rot\tilde{E}_{n}=\rot\hat{E}_{n}\xrightarrow{\,\ltom\,}E^{\rot}\,.
	\end{align*}
	Now the weighted Maxwell estimate from $(i)$ shows that  
	$(\tilde{E}_{n})_{n\in\N}$ is a Cauchy sequence in $\ltmoom$, hence 
	converging to some $\tilde{E}\in\ltmoom$. In addition we have
	\begin{itemize}[leftmargin=0.6cm, itemsep=6pt]
		\item $\dsp\forall\;\Phi\in\cicnom$:\quad
			  $\dsp\scp{\tilde{E}}{\rot \Phi}_{\ltom}
			   \xleftarrow{n\to\infty}
			   \scp{\tilde{E}_{n}}{\rot\Phi}_{\ltom}
			    =\scp{\rot\tilde{E}_{n}}{\Phi}_{\ltom}
			   \xrightarrow{n\to\infty}\scp{E^{\rot}}{\Phi}_{\ltom}\,,$	
		\item $\dsp\forall\;\phi\in\cictom$:\quad
			  $\dsp\scp{\eps \tilde{E}}{\nabla \phi}_{\ltom}
			   \xleftarrow{n\to\infty}
			   \scp{\eps\tilde{E}_{n}}{\nabla\phi}_{\ltom}
			    =-\scp{\div\eps\tilde{E}_{n}}{\phi}_{\ltom}=0\,,$		
		\item $\dsp\forall\;H\in\vharmmotn{\eps}(\Om)$:\quad
			  $\dsp\scp{\tilde{E}}{H}_{\ltmoeps(\Om)}
			   \xleftarrow{n\to\infty}
			   \scp{\tilde{E}_{n}}{H}_{\ltmoeps(\Om)}=0\,,	$
	\end{itemize}
	such that by \eqref{equ:prel_weak=strong_2}
	\begin{align*}
		\tilde{E}
		\in\rmotom
		\cap\eps^{-1}\zdmonom
		\cap\vharmmotn{\eps}(\Om)^{\perp_{-1,\eps}}
		\quad\text{with}\quad
		\rot\tilde{E}=E^{\rot}
	\end{align*}
	and $(b)$ is proven. The last assertion $(c)$ follows by similar 
	arguments.
\end{proof}
\begin{rem}\label{rem:stp_toolbox-reproduced}
	Under the assumptions of Lemma \ref{lem:stp_toolbox-reproduced} we have
	in particular
	\begin{align*}
		\ltom=\ovl{\calR(\div_{\Gan}\eps)}
		      \oplus_{\eps}
		      \calN(\operatorname{\grad}_{\Gat})
		     =\div\dmonom\oplus_{\eps}\{0\}
		     =\div\dmonom
	\end{align*}
 	and by \eqref{equ:stp_decomp} the following Helmholtz type decompositions 
 	hold true:
	\begin{alignat*}{2}
		\ltom &=\nabla\homotom
			\oplus_{\eps}\eps^{-1}\zdznom\,,
		\quad\qquad &
		\ltom &=\eps^{-1}\rot\rmotom
			\oplus_{\eps}\zrznom\,,\\
		\zrztom &=\nabla\homotom
			\oplus_{\eps}\vHarmtn{\eps}(\Om)\,,
		\quad\qquad &
			\eps^{-1}\zdztom &=\eps^{-1}\rot\rmotom
			\oplus_{\eps}\vHarmnt{\eps}(\Om)\,.\\
	\end{alignat*}	
\end{rem}
%
\subsection{Dirichlet-Neumann Fields in Exterior Domains} %
%
As noted before, to solve 
\eqref{equ:stp_problem_es}\;resp.\;\eqref{equ:stp_problem_ms} with Hilbert space 
methods we have to deal with 
$\vharmmotn{\eps}(\Om)$\;\,resp.\;$\vharmmont{\mu}(\Om)$. 
Therefore, a more thorough investigation of these fields is needed. \\[12pt]
From the literature, it is well known, that the existence of 
Dirichlet-Neumann fields is strongly related to the topological properties of the 
domain $\Om$. For example, as shown in \cite{picard_boundary_1982} (\,see also 
\cite{milani_decomposition_1988}\,) in the limit cases 
$\Gat=\Ga$\;resp.\;$\Gat=\emptyset$ the dimension of 
$\vHarmtn{}(\Om)=\vHarmtn{\mathbbm{1}}(\Om)$ is essentially given by the number of 
connected components of the boundary $\Ga$\;resp.\;the number of handles of $\Om$. 
In addition, as in \cite[Lemma 3.8]{pauly_generalized_2009} we obtain 
for $\ga$ $\ka-\co-$decaying with $\ka>0$
\begin{align}\label{equ:stp_dir-neu-fields_equal}
	\harmgen{\ga}{-\frac{3}{2}}{\Gat}{\Gan}(\Om)
		=\vHarmtn{\ga}(\Om)
		=\harmgen{\ga}{<\frac{1}{2}}{\Gat}{\Gan}(\Om)\,,	
\end{align}
and an easy application of the Helmholtz decompositions \eqref{equ:stp_decomp} 
shows that the dimension of the Dirichlet-Neumann fields $\vHarmtn{\ga}(\Om)$ 
does not depend on the transformation $\ga$, i.\,e.,
\begin{align}\label{equ:dnf_transf-ind}
	d_{1,2}:=\dim\vHarmtn{}(\Om)=\dim\vHarmtn{\ga}(\Om)
	=\dim\vharmmotn{\ga}(\Om)<\infty\,.
\end{align}

\begin{figure}[h]
\captionsetup{width=0.7\textwidth}
\centering
\begin{tikzpicture}
	\tikzset{
  		ring shading/.code args={from #1 at #2 to #3 at #4}{
    	\def\colin{#1}
    	\def\radin{#2}
    	\def\colout{#3}
    	\def\radout{#4}
    	\pgfmathsetmacro{\proportion}{\radin/\radout}
    	\pgfmathsetmacro{\outer}{0.8818cm}
    	\pgfmathsetmacro{\inner}{0.8818cm*\proportion}
    	\pgfmathsetmacro{\innerlow}{\inner-0.01pt}
    	\pgfdeclareradialshading{ring}{\pgfpoint{0cm}{0cm}}%
    	{
    	  color(0pt)=(white);
    	  color(\innerlow)=(white);
    	  color(\inner)=(#1);
    	  color(\outer)=(#3)
    	}
    	\pgfkeysalso{/tikz/shading=ring}
  		},
	}
	\begin{scope} 
	    \shade[even odd rule, ring shading={from lightgray at 3.5 to white at 4.495}]
  		(0,0) circle (4.5cm) circle (3.59cm);
		\fill[white] (0,0) circle (3.599cm);
		\fill[lightgray,opacity=0.90] (0,0) circle (3.6cm);
		\draw[dashed,line width=1pt] (0,0) circle (3cm);
		\filldraw[white] (0,0)--(1,0)--(1.5,0.5)
		                     --(1.5,0.5) arc(5:29:2)--(1.24,1.32)
		                     --(0.8,1)--(0.5,1.8)
		                     --(0.505,1.8) arc(270:249.76:2)--(-0.19,1.92)
		                     --(-0.65,1.405)--(-0.65,1.405) arc (40:10:3);
		\draw[line width=1pt](-0.017,0.01)--(1.01,0.01)-- (1.52,0.52);
		\draw[line width=2pt] (1.496,0.5) arc(5:30:2);
		\draw[line width=2pt] (1.265,1.34)--(0.8,1.015)--(0.474,1.822);
		\draw[line width=2pt] (0.505,1.8) arc(270:249.7:2);
		\draw[line width=1pt] (-0.668,1.396)--(-0.165,1.945);
		\draw[line width=1pt] (0,0) arc (10:40.2:3);
		\filldraw[lightgray] (0.2,1.2) circle (0.22cm);
		\draw[line width=1pt] (0.2,1.2) circle (0.22cm);
		\filldraw[white] (-1.95,-1.14) arc (-12.8:10:2.75)
							 --(-1.91,0) arc (80:57.2:2.7)
							 --(-0.5,-1)--(-0.85,-0.43)
							 --(-0.5,-1) arc(110:79:1.5)
							 --(0.3,-0.94)--(0.7,-2)--(0,-1.6)
							 --(-0.5,-1.8)--(-1,-1.4)--(-1.5,-1.8)
							 --(-1.96,-1.1);
		\draw[line width=2pt] (-1.95,-0.01) arc (80:57.2:3);
		\draw[line width=2pt] (-1.921,-0.038) arc (10:-12.8:2.75);
		\draw[line width=2pt] (-0.499,-1.02)--(-0.86,-0.43);
		\draw[line width=1pt] (-0.51,-1) arc(110.5:79:1.5);
		\draw[line width=1pt] (0.288,-0.922)--(0.7,-2)--(0,-1.6)--
							  (-0.5,-1.8)--(-1.02,-1.4)--(-1.485,-1.78);
		\draw[line width=2pt] (-1.5,-1.81)--(-1.952,-1.099);
		\filldraw[lightgray] (-1.2,-0.9) circle (0.22cm);
		\draw[line width=1pt] (-1.2,-0.9) circle (0.22cm);
		\filldraw[lightgray] (-0.2,-1.3) circle (0.12cm);
		\draw[line width=2pt] (-0.2,-1.3) circle (0.12cm);
		\filldraw[white] (1.4,-0.635) arc (110:70:0.9)
							 --(2.06,-0.64) arc (160:200:0.9)
							 --(1.97,-1.25) arc (288:255:0.9);
							 --(1.4,-0.6) arc (20:-20:0.9);	
		\draw[line width=2pt] (1.4,-0.635) arc (110:70:1);
		\draw[line width=2pt] (1.4,-0.6) arc (20.1:-20.1:1);
		\draw[line width=2pt] (1.4,-1.25) arc (250:290:1);
		\draw[line width=2pt] (2.09,-0.6) arc (160:200:1);
		\draw (-1.5,1) node {$\Omrhat$};
		\draw (0.6,0.6) node {$\rthree\sm\Om$};
		\draw (2.8,-1.9) node {$\Sp_{\rhat}$};
		\draw[lightgray] (-3.2,-2.6) node {$\Om$};
		\draw (0.9,1.8) node {$\Gat$};
		\draw (-2.2,-0.6) node {$\Gat$};
		\draw (1.2,-0.9) node {$\Gat$};
		\draw (0.1,-1.95) node {$\Gan$};
		\draw (-0.3,0) node {$\Gan$};
	\end{scope}
\end{tikzpicture}
\caption{$\rthree\sm\Om$ surrounded by the boundary parts $\Gat$ (\,thick 
black lines\,) and $\Gan$ (\,thin black lines\,) as well as the artificial 
boundary sphere $\Sp_{\rhat}$ (\,dashed line\,).}
\end{figure}
As a crucial technical trick we will show that there exists a finite set of 
compactly supported vector fields $\BGatom$, whose projections form a basis of 
$\vHarmtn{\ga}(\Om)$. The underlying idea is, that $\Omrhat$ and $\Om$ have 
essentially the same topological properties. Hence, choosing a basis of 
$\vHarmv{\ga}{\Ga_{1,\rhat}}{\Gan}(\Omrhat)$, 
extending their elements by zero to $\Om$ and projecting them onto 
$\vHarmtn{\ga}(\Om)$, 
we obtain a basis of $\vHarmtn{\ga}(\Om)$. Moreover, the extensions by zero define 
exactly the set $\BGatom$, which will also serve as a set of linear functionals 
ensuring uniqueness of static solutions.
\begin{theo}\label{thm:stp_dir-neu-char}
	There exist a finite set
	\begin{align*}
		\BGatom=\{\scrB_{1,1},\scrB_{1,2},\ldots,
				  \scrB_{1,d_{1\hspace*{-0.025cm},\hspace*{-0.02cm}2}}\}
		\subset\zrztom
		\qquad\text{with}\qquad
		\vHarmtn{\ga}(\Om)\cap\BGatom^{\perp_{\ga}}=\{0\}\,.
	\end{align*}
	In addition, the elements of $\BGatom$ have compact support and their 
	projections {\rm (}\,in $\ltgaom$\,{\rm)} along $\ovl{\nabla\Hoztom}$ form a 
	basis of the Dirichlet-Neumann fields $\vHarmtn{\ga}(\Om)$.
\end{theo}
\noindent
\begin{proof}
	The proof is given in the Appendix.
\end{proof}
\noindent
Note that, as $\BGatom$ contains only compactly supported functions, we 
obviously have 
\begin{align}\label{equ:stp_dir-neu-gen_int}
   	  \forall\;s\in\reals:\qquad\qquad
	  	\ovl{\rot\rsmonom}^{\,\norm{\cdot}_{\ltsom}}
	  	\cup\,
	  	\ovl{\rot\Rsnom}^{\,\norm{\cdot}_{\ltsom}}
	  	&\subset\,\BGatom^{\perp}\,.
\end{align}
Therefore, $\BGatom$ allows for an alternative characterization for 
$\ovl{\calR(\rot_{\Gan})}$ and, in particular, we may generalize 
the weighted Maxwell estimate from Lemma \ref{lem:stp_toolbox-reproduced}.
\begin{lem}\label{lem:stp_max-est_gen}  
	Let $\eps$ be $\ka-\co-$decaying with order $\ka>0$ and 	$\BGatom$ be the
	finite set from Theorem \ref{thm:stp_dir-neu-char}.\\[-10pt]
	\begin{enumerate}[itemsep=2pt,label=$(\roman*)$]
		\item It holds
			  \begin{align*}
				  \eps^{-1}\zdznom\cap\BGatom^{\perp_{\eps}}
					&=\eps^{-1}\zdznom\cap\vHarmtn{\eps}(\Om)^{\perp_{\eps}}
					=\eps^{-1}\ovl{\rot\Rznom}\,.
		  	  \end{align*}
		\item There exists $c>0$ such that for all 
			  $E\in\rmotom\cap\eps^{-1}\dmonom$ it holds
		  	  \begin{align*}
		  		  \normltmoom{E}
					  &\leq c\,\Big(\,\normltom{\rot E}	
				 	  +\normltom{\div\eps E}
			     	  +\dsp\sum_{\ell=1,\ldots,
			     	   d_{1\hspace*{-0.025cm},\hspace*{-0.02cm}2}}
			       	   |\scp{E}{\scrB_{1,\ell}}_{\lteps(\Om)}|\;\Big)\,.
		 	  \end{align*} 
	\end{enumerate}
\end{lem}
\begin{proof}
	By \eqref{equ:stp_complex}, \eqref{equ:stp_decomp} and 
	\eqref{equ:stp_dir-neu-gen_int} we clearly have
	\begin{align*}
		\eps^{-1}\zdznom\cap\vHarmtn{\eps}(\Om)^{\perp_{\eps}}
			=\eps^{-1}\ovl{\rot\Rznom}
			\subset\eps^{-1}\zdznom\cap\BGatom^{\perp_{\eps}}\,.
	\end{align*}
	Now let $E\in\eps^{-1}\zdznom\cap\BGatom^{\perp_{\eps}}$. 
	Then, by \eqref{equ:stp_decomp}, we decompose
	\begin{align*}
		E=\calE+\calH
		  \in\eps^{-1}\ovl{\rot\Rznom}
		  \oplus_{\eps}
		  \vHarmtn{\eps}(\Om)\,,	
	\end{align*}
	hence, by \eqref{equ:stp_dir-neu-gen_int} and 
	Theorem \ref{thm:stp_dir-neu-char} we have 
	$\calH=E-\calE\in\vHarmtn{\eps}(\Om)\cap\BGatom^{\perp_{\eps}}=\{0\}$, which 
	proves statement $(i)$. In order to show $(ii)$ we assume the estimate to be 
	wrong. Then there exists  
	\begin{align*}
	 	(E_{n})_{n\in\mathbb{N}}\subset\rmotom\cap\eps^{-1}\dmonom
	 	\quad\text{with}\quad	
		\normltmoom{E_{n}}=1
	\end{align*}
	and
	\begin{align*}
		\rot E_{n}\xrightarrow{\;\ltom\;} 0
		\,,\qquad\qquad
		\div\eps E_{n}\xrightarrow{\;\ltom\;} 0	
		\,,\qquad\qquad
		\scp{E_{n}}{\scrB_{1,\ell}}_{\lteps(\Om)}
		\xrightarrow{\;\;\;\,\C\;\;\;\,} 0
		\qquad
		(\,\ell=1,\ldots,d_{1,2}\,)
	\end{align*}
	for $n\To\infty$. Thus $(E_{n})_{n\in\N}$ is bounded in 
	$\rmotom\cap\eps^{-1}\dmonom$ and by Weck's local selection theorem it has a 
	subsequence $(E_{\pi(n)})_{n\in\N}$ converging in $\ltlocomb$. By 
	Lemma \ref{lem:stp_max-est_mud}, this sequence even converges in 
	$\rmotom\cap\eps^{-1}\dmonom$ to some 
	\begin{align*}
		E\in\rmotom\cap\eps^{-1}\dmonom
		\quad\text{with}\quad
		\rot E=0\quad\text{resp.}\;\,\div\eps E=0\,.\\[-11pt]
	\end{align*}
	We obtain $E\in\vharmmotn{\eps}(\Om)
	\overset{\eqref{equ:stp_dir-neu-fields_equal}}{=}\vHarmtn{\eps}(\Om)$ and 
	additionally
	\begin{align*}
		\scp{E}{\scrB_{1,\ell}}_{\lteps(\Om)}
		  =\lim_{n\to\infty}\scp{E_{\pi(n)}}{\scrB_{1,\ell}}_{\lteps(\Om)}=0
		  \,,\qquad\qquad
		  \ell=1,\ldots,d_{1,2}\,,
	\end{align*}
	hence $E\in\vHarmtn{\eps}(\Om)\cap\BGatom^{\perp_{\eps}}=\{0\}$ by 
	Theorem \ref{thm:stp_dir-neu-char}, a contradiction.  
\end{proof}
\begin{rem}
	Note that in Theorem \ref{thm:stp_dir-neu-char} and Lemma 
	\ref{lem:stp_max-est_gen} (i) no assumption on $\ga$ resp. $\eps$ is required, 
	except of the General Assumption \ref{gen-ass}. 
\end{rem}
\vspace*{12pt}
%
\subsection{Static Solution Theory} 						 %
%
Let us turn back to the boundary value problem of electro-magneto-statics, using 
\eqref{equ:stp_problem_es} as an illustrative example. As indicated by Lemma 
\ref{lem:stp_toolbox-reproduced} we will solve \eqref{equ:stp_problem_es} for 
given data $(G,f)\in\ltom\times\ltom$ by constructing a solution $E\in\ltmoom$. 
In order to obtain uniqueness, we have to impose some additional conditions, but 
instead of projecting to Dirichlet-Neumann fields, we use projections to 
$\BGatom$.
\begin{defi}
	Let $(G,f,\zeta)\in\ltlocomb\times\ltlocomb
	\times\C^{d_{1\hspace*{-0.025cm},\hspace*{-0.02cm}2}}$. We call 
	$E$ ``(static) solution'' of \eqref{equ:stp_problem_es}, if
	\begin{align*}
		E\in\rgen{}{}{-1,\Gat}(\Om)
		    \cap\eps^{-1}\dgen{}{}{-1,\Gan}(\Om)	
	\end{align*}
	satisfies
	\begin{align}
		\rot E=G
		\,,\qquad\qquad
		\div\eps E=f
		\,,\qquad\qquad 
		\scp{E}{\scrB_{1,\ell}}_{\lteps(\Om)}=\zeta_{\ell}
		\quad\quad(\,
		\ell=1,\ldots,d_{1,2}\,)\,,
	\end{align}
	where $\{\,\scrB_{1,1},\scrB_{1,2},\ldots,
	\scrB_{1,d_{1\hspace*{-0.025cm},\hspace*{-0.02cm}2}}\,\}$ are the 
	elements in $\BGatom$ from Theorem 
	\ref{thm:stp_dir-neu-char}.
\end{defi}
\noindent
Let $G\in\ltom$, $f\in\ltom$, 
$\zeta\in\C^{d_{1\hspace*{-0.025cm},\hspace*{-0.02cm}2}}$,
and let $\eps$ decay with order $\ka>0$. First of all note that 
\eqref{equ:stp_problem_es} admits at most one \emph{static solution}, as for 
the homogeneous problem 
$E\in\vharmmotn{\eps}(\Om)\cap\BGatom^{\perp_{\eps}}$ 
together with \eqref{equ:stp_dir-neu-fields_equal} and Theorem 
\ref{thm:stp_dir-neu-char} yields $E=0$. Turning to existence, necessary 
conditions are obviously 
\begin{align*}
	G\in\rot\rmotom\qquad\text{and}\qquad f\in\div\dmonom,
\end{align*}
the latter one being no further restriction as by Lemma 
\ref{lem:stp_toolbox-reproduced}, Remark \ref{rem:stp_toolbox-reproduced} we 
have $\div\dmonom=\ltom$. But in fact this conditions are already sufficient 
since Lemma \ref{lem:stp_toolbox-reproduced} also yields  
\begin{align*}
	E_{1}\in\rmotom\cap\eps^{-1}\zdmonom
	\qquad\text{and}\qquad
	E_{2}\in\dmonom\cap\eps\,\zrmotom
\end{align*}
with $\rot E_{1}=G$ and $\div E_{2}=f$. Thus, 
\begin{align*}
	\widehat{E}
		:=E_{1}+\eps^{-1}E_{2}
		  \in\rmotom\cap\eps^{-1}\dmonom
\end{align*}
already satisfies 
\begin{align*}
	\rot\widehat{E}=\rot E_{1}=G
	\qquad\text{and}\qquad
	\div\eps\widehat{E}=\div E_{2}=f	\,.
\end{align*}
Moreover, assuming we are able to construct 
$H\in\vharmmotn{\eps}(\Om)=\vHarmtn{\eps}(\Om)$ with 
\begin{align}\label{equ:stp_sol-the_unique-cond}
	\quad\scp{H}{\scrB_{1,\ell}}_{\lteps(\Om)}
		=\zeta_{\ell}-\scp{\widehat{E}}{\scrB_{1,\ell}}_{\lteps(\Om)}
		:=\tilde{\zeta}_{\ell}
	\,,\qquad\qquad
	\ell=1,\ldots,d_{1,2}\,,	
\end{align}
the sum 
\begin{align*}
	E:=\widehat{E}+H\in\rmotom\cap\eps^{-1}\dmonom	
\end{align*}
solves
\begin{align*}
	\rot E=G
	\,,\qquad\qquad
	\div\eps E=f
	\,,\qquad\qquad 
	\scp{E}{\scrB_{1,\ell}}_{\lteps(\Om)}=\zeta_{\ell}
	\quad\quad
	(\,\ell=1,\ldots,d_{1,2}\,)\,,
\end{align*}
hence $E$ is a static solution of \eqref{equ:stp_problem_es}. It remains to 
construct $H$ such that \eqref{equ:stp_sol-the_unique-cond} 
holds. For that we decompose $\scrB_{\ell}$ according to Remark 
\ref{rem:stp_toolbox-reproduced} in
\begin{align*}
	\scrB_{1,\ell}
		=\nabla w_{\ell}+H_{\ell}
		 \in\nabla\homotom\oplus_{\eps}\vHarmtn{\eps}(\Om)
	\,,\qquad\qquad
	\ell=1,\ldots,d_{1,2}\,,
\end{align*}
noting that by Theorem \ref{thm:stp_dir-neu-char} 
$\{H_{\ell}\}_{\ell}$ is a basis of $\vHarmtn{\eps}(\Om)$ 
and w.l.o.g.\;orthonormal in $\lteps(\Om)$. Then
\begin{align*}
	H:=\sum_{j=1,\ldots,d_{1,2}}
	       \tilde{\zeta}_{j}H_{j} \in\vHarmtn{\eps}(\Om)
\end{align*}
indeed satisfies 
\begin{align*}
	\scp{H}{\scrB_{1,\ell}}_{\lteps(\Om)}
		=\underbrace{\scp{H}{\nabla w_{\ell}}_{\lteps(\Om)}}_{=\,0}
		 +\sum_{j=1,\ldots,d_{1\hspace*{-0.025cm},\hspace*{-0.02cm}2}}
		  \tilde{\zeta}_{j}\,\scp{H_{j}}{H_{\ell}}_{\lteps(\Om)}
		=\tilde{\zeta}_{\ell}\,,\quad\quad\ell=1,\ldots,d_{1,2}\,.
\end{align*}
and we have solved the electro-static problem \eqref{equ:stp_problem_es}.
\begin{theo}\label{thm:stp_sol-thm-es}
	Let $\eps$ be 
	$\ka-\co-$decaying with $\ka>0$. For all $(G,f)\in\ltom\times\ltom$ with 
	\begin{align*}
		G\in{}_{0}\mathbb{D}{}_{\Gat}(\Om):=\zdztom\cap\BGanom^{\perp}
	\end{align*}
	and $\zeta\in\C^{d_{1,2}}$ there exists a unique static solution 
	\begin{align*}
		E\in\rmotom\cap\eps^{-1}\dmonom
	\end{align*}
	of \eqref{equ:stp_problem_es}. In addition, the corresponding solution 
	operator
	\begin{align*}
		\Map{\calL_{\eps,0}}
		  	 {{}_{0}\mathbb{D}{}_{\Gat}(\Om)\times\ltom\times\C^{d_{1,2}}}
			 {\rmotom\cap\eps^{-1}\dmonom}
		  	 {(G,f,\zeta)}{E}	
	\end{align*}
	is continuous.
\end{theo}
\begin{proof}
	It remains to show that $\calL_{\eps,0}$ is bounded. But this is a direct 
	consequence of Lemma \ref{lem:stp_max-est_gen}, $(ii)$.
\end{proof}
\noindent
Swapping $\Gat$ and $\Gan$\;resp.\;$\eps$ and $\mu$ we obtain a 
corresponding result for the magneto-static problem \eqref{equ:stp_problem_ms}. 
\begin{theo}\label{thm:stp_sol-thm-ms}
	Let $\mu$ be 
	$\ka-\co-$decaying with $\ka>0$. For all $(F,g)\in\ltom\times\ltom$ with 
	\begin{align*}
		F\in{}_{0}\mathbb{D}{}_{\Gan}(\Om):=\zdznom\cap\BGatom^{\perp}
	\end{align*}
	and $\theta\in\C^{d_{2\hspace*{-0.025cm},\hspace*{-0.02cm}1}}$ there exists 
	a unique static solution 
	\begin{align*}
		H\in\rmonom\cap\mu^{-1}\dmotom
	\end{align*}
	of \eqref{equ:stp_problem_ms}. In addition, the corresponding solution 
	operator
	\begin{align*}
		\Map{\calL_{\mu,0}}
		  	 {{}_{0}\mathbb{D}{}_{\Gan}(\Om)\times
		  	 \ltom\times\C^{d_{2\hspace*{-0.025cm},\hspace*{-0.02cm}1}}}
			 {\rmonom\cap\mu^{-1}\dmotom}
		  	 {(F,g,\theta)}{H}	
	\end{align*}
	is continuous.
\end{theo}
\begin{rem}
	By Theorem 
	\ref{thm:stp_sol-thm-es} and Theorem \ref{thm:stp_sol-thm-ms} for all 
	\begin{align*}
	  \big(F,g,G,f,\zeta,\theta\big)
	  \in{}_{0}\mathbb{D}{}_{\Gan}(\Om)
	  \times
	  \ltom
	  \times
	  {}_{0}\mathbb{D}{}_{\Gat}(\Om)
	  \times
	  \ltom
	  \times
	  \C^{d_{1\hspace*{-0.025cm},\hspace*{-0.02cm}2}}
	  \times
	  \C^{d_{2\hspace*{-0.025cm},\hspace*{-0.02cm}1}}	
	\end{align*}
	the electro-magneto static system \eqref{equ:stp_problem_es}, 
	\eqref{equ:stp_problem_ms} has a unique solution
	\begin{align*}
		(E,H)\in
			 \big(\rmotom\cap\eps^{-1}\dmonom\big)
		     \times
		     \big(\rmonom\cap\mu^{-1}\dmotom\big)\,.
	\end{align*}
	The corresponding solution operator is continuous and will be denoted 
	by $\calL_{\La,0}$
\end{rem}
%
%
\section{The Time-Harmonic Problem $\om\neq0$}
\label{sec:time-harmonic}
%
%
Having established the static solution theory we treat the time-harmonic 
case. For sake of brevity we just concentrate on the main results and refer to 
\cite{osterbrink_time-harmonic_2019} for the details and some additional results.
Let  
\begin{align*}
	\om\in\C_{+}:=\set{z\in\mathbb{C}}{\Im(z)\geq 0}
	\qquad\text{with}\qquad
	\om\neq 0\,.	
\end{align*}
We are looking for an electro-magnetic field 
$(E,H)\in\rloctom\times\rlocnom$ such that for given data 
$(F,G)\in\ltlocomb\times\ltlocomb$ it 
holds
\[(\,\M+i\om\La\,)\,(E,H)=(F,G)\,.\]
By \eqref{equ:prel_weak=strong_1} the \emph{``Maxwell-operator''}  
\begin{align*}
	\maps{\scrM}
		{\Rztom\times\Rznom\subset\lteps(\Om)\times\ltmu(\Om)}
		{\lteps(\Om)\times\ltmu(\Om)}	
		{(E,H)}{i\La^{-1}\M\hspace*{0.02cm}(E,H)}\,,
\end{align*}
is self-adjoint which in the case of $\om\in\C\sm\reals$ immediately yields an 
$\lt$-solution theory.
\begin{theo}\label{thm:thp_l2-sol}
	Let $\om\in\C\sm\reals$. For every $(F,G)\in\ltom\times\ltom$ 
	system \eqref{equ:int_max-sys}, \eqref{equ:int_bd-cond} has a unique solution 
	\[(E,H)\in\Rztom\times\Rznom\,.\] Moreover, the solution operator, 
	which we denote by $\calL_{\La,\om}:=i(\,\scrM-\om\,)^{-1}\La^{-1}$ 
	is continuous.
\end{theo}  
The case $\om\in\reals\sm\{0\}$ is more challenging, since we want to solve in 
the continuous spectrum of $\scrM$. Clearly this cannot be done for every 
$(F,G)\in\ltom\times\ltom$, since otherwise $\om\not\in\sigma(\scrM)$. Thus 
we have to work in certain subspaces of $\ltom\times\ltom$ and we have to 
generalize the solution concept. 
\begin{defi}
	Let $\om\in\reals\sm\{0\}$ and $(F,G)\in\ltlocom\times\ltlocom$. We call 
	$(E,H)$\;``(radiating) solution"\;of the time-harmonic boundary value 
	problem \eqref{equ:int_max-sys}, \eqref{equ:int_bd-cond}, if
	\begin{align*}
		(E,H)\in\Rsmtom{-\frac{1}{2}}\times\Rsmnom{-\frac{1}{2}}	
	\end{align*}	
	and satisfies
	\begin{align}\label{equ:thp_rad-sol_rad-cond}
		\big(\,\M+i\om\La\,\big)\hspace*{0.03cm}(E,H)=(F,G)
		\,,\qquad
		\big(\,\Laz+\sqrt{\eps_{0}\mu_{0}}\;\Xi\;\big)\hspace*{0.03cm}(E,H)
		\in\ltbigom{-\frac{1}{2}}\times\ltbigom{-\frac{1}{2}}\,,
	\end{align}
	where
	\begin{align*}
		\Laz:=\ptwomat{\eps_{0}}{0}{0}{\mu_{0}}
		\qquad\text{and}\qquad	
		\Xi:=\ptwomat{0}{-\xi\times}{\xi\times}{0}\,.
	\end{align*}
\end{defi}
Conveniently, we can apply the same methods as in 
\cite{pauly_low_2006}, see also \cite{picard_time-harmonic_2001, 
weck_complete_1992, weck_generalized1_1997}, to obtain a solution 
theory. In particular, we use the limiting absorption principle 
introduced by Eidus and approximate solutions to 
$\om\in\reals\sm\{0\}$ by solutions corresponding to $\om\in\C_{+}\sm\reals$. 
Again, Weck's local selection theorem is the crucial tool in the limit process. 
Additionally, the polynomial decay of eigenfunctions as well as an a-priori 
estimate for solutions corresponding to non-real frequencies are needed and both 
are obtained by reduction to similar results known for the Helmholtz equation in 
the whole of $\rthree$. For the details see \cite{osterbrink_time-harmonic_2019}. 
\begin{theo}[\,Generalized Fredholm Alternative\,]\label{thm:thp_fredh-alt}
	Let $\om\in\reals\sm\{0\}$ and 
	let $\eps$,\,$\mu$ be $\ka$-decaying with $\ka>1$. Moreover, let 	
	\begin{align*}
	\gk{\,\scrM-\om\,}
		&:=\setb{(E,H)}
		   {(E,H)\text{ is a radiating solution of }(\,\M+i\om\La\,)\,(E,H)=0\,}\,,\\
		\gs &:=\setb{\,\om\in\C\sm\{0\}\hspace*{0.03cm}}
			   {\,\gk{\,\scrM-\om\,}\neq\{0\}\,}\,.
	\end{align*}
	Then:\\[-10pt]
	\begin{enumerate}[label=$(\roman*)$]
		\item For all $t\in\reals$
			  \begin{align*}
			  	\gk{\,\scrM-\om\,}
			  	\subset
			  	\Big(\,\Rttom\cap\eps^{-1}\rot\Rtnom\,\Big)
			   	\,\times\,
			   	\Big(\,\Rtnom\cap\mu^{-1}\rot\Rttom\,\Big)\,.
			  \end{align*}
		\item $\dim\,\gk{\,\scrM-\om\,}<\infty$\,.\\[-4pt]
		\item $\gs\subset\reals\sm\{0\}$ and $\gs$ has no accumulation point in 
			  $\reals\sm\{0\}$\,.\\[-4pt]
		\item For all
			  $(F,G)\in\dsp\ltbigom{\frac{1}{2}}\times\ltbigom{\frac{1}{2}}$ 
			  there exists a radiating solution $(E,H)$ of 
			  \eqref{equ:int_max-sys}, \eqref{equ:int_bd-cond}, if and only if
			  \begin{align*}
			  		\forall\;(e,h)\in\gk{\,\scrM-\om\,}:
			  		\qquad\scpltom{(F,G)}{(e,h)}=0\,.	
			  \end{align*}
			  Moreover, the solution $(E,H)$ can be chosen, such that
			  \begin{align*}
			  		\forall\;(e,h)\in\gk{\,\scrM-\om\,}:
			  		\qquad\scpltLaom{(E,H)}{(e,h)}=0\,.
			  \end{align*}
			  Then $(E,H)$ is uniquely determined.\\[-4pt]
		\item For all $s$,${-t}>1/2$ the solution operator
			  \begin{align*}
			  	\hspace*{0.75cm}
			  	\map{\calL_{\La,\om}}
			  	    {\Big(\,\ltsom\times\ltsom\,\Big)
			  	     \cap\gk{\,\scrM-\om\,}^{\perp}}
			  		{\Big(\,\Rttom\times\Rtnom\,\Big)\cap\,
			  	     \gk{\,\scrM-\om\,}^{\perp_{\La}}}
			  \end{align*}
			  defined by $(\mathrm{4})$ is continuous. Here $\perp_{\La}$ 
			  indicates the orthogonal complement in $\ltLaom$.\\[-2pt]
	\end{enumerate}
\end{theo}
\begin{rem}\label{rem:thp_sol-op}
	By 
	Theorem \ref{thm:thp_l2-sol} and Theorem \ref{thm:thp_fredh-alt} 
	and for all
	\begin{align*}
		(F,G)
		\in\Big(\,\ltbig{\frac{1}{2}}(\Om)\times\ltbig{\frac{1}{2}}(\Om)\,\Big)
		   \cap\gk{\,\scrM-\om\,}^{\perp}
	\end{align*} 
	the time-harmonic Maxwell system \eqref{equ:int_max-sys}, 
	\eqref{equ:int_bd-cond} has a unique radiating solution 
	\begin{align*}
		(E,H)\in\Rsmtom{-\frac{1}{2}}\times\Rsmnom{-\frac{1}{2}}
		\quad\text{with}\quad
		\big(\,\Laz+\sqrt{\eps_{0}\mu_{0}}\;\Xi\;\big)\hspace*{0.03cm}(E,H)
		\in\ltbigom{-\frac{1}{2}}\times\ltbigom{-\frac{1}{2}}\,.
	\end{align*}
	The corresponding solution operator is continuous and will 
	be denoted by $\calL_{\La,\om}$.
\end{rem} 
%
%
\section{Low Frequency Asymptotics $\om\to0$}
\label{sec:low-frequency-asymptotics}
%
%
In order to discuss the low frequency asymptotics we first have to ensure that 
$\gs$ does not accumulate at zero. For that we show an estimate emerging from a 
representation formula for the homogeneous, isotropic whole space problem, i.e., 
$\Om=\rthree$ and $\Lambda=\Laz$.
\begin{pro}
	$\Om=\rthree$, $\La=\Laz$ and 
	$\om\in\C_{+}\sm\{0\}$, it holds 
	\[\gk{\,\scrM-\om\,}=\{0\}\,.\]
	Thus the solution operator $\calL_{\Laz,\om}$ is well 
	defined for all $(F,G)\in\ltbig{\frac{1}{2}}\times\ltbig{\frac{1}{2}}$.
\end{pro}
\begin{proof}
	Let $(E,H)\in\gk{\,\scrM-\om\,}$. By Theorem 
	\ref{thm:thp_fredh-alt} (i) and the differential equation we have 
	\begin{align*}
		(E,H)\in(\,\R\cap\zd\,)
			\times
			(\,\R\cap\zd\,)
		\qquad\text{with}\qquad
		\M\,(E,H)=-i\om\Lambda_{0}(E,H)\,.
	\end{align*}
	Hence, by \cite[Lemma 4.2]{kuhn_regularity_2010}, 
	$(E,H)\in(\,\Hgen{}{k}{}\cap\zd\,)\times(\,\Hgen{}{k}{}\cap\zd\,)$ for all 
	$k\in\N_{0}$ and we obtain
	\begin{align*}
		\Delta\,(E,H)=\M^2(E,H)=-\om^2\eps_{0}\mu_{0}\,(E,H)\,.
	\end{align*}
	In other words, $(E,H)\in\Htwo\times\Htwo$ satisfies the Helmholtz-equation 
	with right hand side zero. If $\om\in\C\sm\reals$ we are done, since 
	$\map{\Delta}{\Htwo\subset\lt}{\lt}$ is selfadjoint and therefore 
	$\sigma(\Delta)\subset\reals$, yielding $(E,H)=(0,0)$. For 
	$\om\in\reals\sm\{0\}$ the assertion follows using the Rellich estimate 
	(\,cf. \cite{leis_initial_2013}, p.\,59\,) and the unique continuation 
	principle.
\end{proof}
Now, let $\Om=\rthree$, $\Lambda=\Laz$, $\om\in\C_{+}\sm\{0\}$, 
$(F,G)\in\cic\times\cic$, and let $(E,H):=\calL_{\Laz,\om}\,(F,G)\,$ be the 
corresponding radiating solution. Again, by 
\cite[Lemma 4.2]{kuhn_regularity_2010} and the differential equation, we obtain 
\begin{align*}
	(E,H)\in\big(\,\Hgen{}{2}{<-\frac{1}{2}}\cap\ci\,\big)
	\times
	\big(\,\Hgen{}{2}{<-\frac{1}{2}}\cap\ci\,\big) 
	\qquad\text{and}\qquad 
	\big(\,\Delta+\eps_{0}\mu_{0}\,\om^2\,\big)\,(E,H)
	=(\hat{F},\hat{G})\in\cic\times\cic\,,
\end{align*}
where
\begin{align}\label{equ:lfa_hh-sol}
	(\hat{F},\hat{G}):=(\,\M-i\om\tLaz\,)\,(F,G)-\frac{i}{\om}\Laz^{-1}(\nabla\div F,\nabla\div G)\,,
	\qquad\quad
	\tLaz:=\ptwomat{\mu_{0}}{0}{0}{\eps_{0}}\,.
\end{align}
In fact, $(E,H)$ is the unique radiating solution of the whole space problem 
(\,cf. \cite[Section 4]{weck_generalized1_1997}\,)
\begin{align*}
	&\hspace*{1.5cm}(E,H)\in\Hgen{}{2}{<-\frac{1}{2}}
	 \times\Hgen{}{2}{<-\frac{1}{2}}\,,\\
	&\hspace*{1cm}(\,\Delta+\om^2\eps_{0}\mu_{0}\,)\,(E,H)=(\hat{F},\hat{G})\,,\\
	&\exp\big(\hspace*{-0.03cm}-i\om\sqrt{\eps_{0}\mu_{0}}\,r\,\big)\,(E,H)
	\in\hgen{}{1}{>-\frac{3}{2}}\times\hgen{}{1}{>-\frac{3}{2}}\,.
\end{align*}
For non-real frequencies $\om\in\C_{+}\sm\reals$ this is trivial, because then
\cite[Lemma 4.2]{kuhn_regularity_2010} yields $(E,H)\in\Htwo\times\Htwo$ and the 
Laplacian is self-adjoint on $\Htwo\times\Htwo$. For $\om\in\reals\sm\{0\}$ the 
radiation condition \eqref{equ:thp_rad-sol_rad-cond} shows 
\[\dsp\,(\xi\cdot E,\xi\cdot H)\in\ltbig{-\frac{1}{2}}\times\ltbig{-\frac{1}{2}}\] 
and via the differential equation and the radiation condition we obtain
\begin{itemize}[leftmargin=0.75cm,]
	\item[] $\rot\big(\hspace*{0.02cm}\exp\big(\hspace*{-0.03cm}
	       -i\om\sqrt{\eps_{0}\mu_{0}}\,r\,\big)\,E\,\big)
	       =\exp\big(\hspace*{-0.03cm}
	        -i\om\sqrt{\eps_{0}\mu_{0}}\,r\,\big)\,
	        \Big(G-i\om\hspace*{0.02cm}\big(\,\mu_{0} H+\sqrt{\eps_{0}\mu_{0}}
	        \;\xi\times E\,\big)\,\Big)
	        \hspace*{-0.03cm}\in\ltbig{-\frac{1}{2}}$\,,
	\item[] $\div\big(\hspace*{0.02cm}\exp\big(\hspace*{-0.03cm}
		   -i\om\sqrt{\eps_{0}\mu_{0}}\,r\,\big)\,E\,\big)
		   =-i\exp\big(\hspace*{-0.03cm}
	        -i\om\sqrt{\eps_{0}\mu_{0}}\,r\,\big)
	 		\,\Big(\,\om\sqrt{\eps_{0}\mu_{0}}\;\xi\cdot E
	 		+(\om\eps_{0})^{-1}\div F\,\Big)\hspace*{-0.03cm}
	 		\in\ltbig{-\frac{1}{2}}$\,.
\end{itemize}
Analogously, we see the corresponding results for $H$. Hence, by 
\cite[Lemma 4.2]{kuhn_regularity_2010},
\begin{align*}
	\exp\big(\hspace*{-0.03cm}-i\om\sqrt{\eps_{0}\mu_{0}}\,r\,\big)\,(E,H)
	\in\hgen{}{1}{>-\frac{3}{2}}\times\hgen{}{1}{>-\frac{3}{2}}\,.
\end{align*}
Thus, by \cite[Theorem 4.27, Remark 4.28]{leis_aussenraumaufgaben_1974} we may 
describe $(E,H)$ using the representation formula of the 
Helmholtz-equation, i.e.,
\begin{align*}
	E=\phi_{\om}\star\hat{F}
	 :=(\,\phi_{\om}\ast\hat{F}
	  _{\ell}\,)_{\substack{\phantom{}\\[1pt]\ell=1,2,3}}
	\,,\qquad\qquad
	H=\phi_{\om}\star\hat{G}
	 :=(\,\phi_{\om}\ast\hat{G}
	  _{\ell}\,)_{\substack{\phantom{}\\[1pt]\ell=1,2,3}}\,,
\end{align*}
where $\phi_{\om}=-(4\pi r)^{-1}\exp\big(\hspace*{-0.03cm}-
i\om\sqrt{\eps_{0}\mu_{0}}\,r\,\big)$
is the fundamental solution of the scalar Helmholtz-equation and $\ast$ denotes 
scalar convolution in $\rthree$. Then \eqref{equ:lfa_hh-sol} yields

\begin{align*}
	E=\phi_{\om}\star\Big(-\rot G
		    -i\om\mu_{0}\,F
		    -\frac{i}{\om\eps_{0}}\nabla\div F\,\Big)
	\,,\;\;\;\;\;\,
	H=\phi_{\om}\star\Big(\rot F
		   -i\om\eps_{0}\,G
		   -\frac{i}{\om\mu_{0}}\nabla\div G\,\Big)\,,
\end{align*}
a representation formula for $(E,H)$ provided $(F,G)\in\cic\times\cic$. Next we 
would like to allow more general right hand sides $(F,G)$. For that we move some 
of the differential operators from $F$\;resp.\;$G$ to $\phi_{\om}$, illustrating 
the procedure for $\phi_{\om}\star\rot F$ and $\phi_{\om}\star\nabla\div F$.
\\[12pt]
As both fields $F$ and $G$ are compactly supported we do not have to worry about
integrability of $\phi_{\om}$ at infinity. In $\U_{1}$ we can estimate 
$|\phi_{\om}|\leq c\cdot r^{-1}$ and $|\nabla\phi_{\om}|\leq c\cdot r^{-2}$, 
hence $\phi_{\om}$,\,$\nabla\phi_{\om}\in\lo(\U_{1})$. Moreover, with 
$\tilde{\eta}$ (\,the cut-off function from above\,) we define for 
$n\in\mathbb{N}$ and fixed $x\in\rthree$ the functions 
\begin{align*}
	\eta_{n}(y):=\tilde{\eta}
	\hspace*{0.02cm}(\hspace*{0.02cm}n\cdot|x-y|\hspace*{0.02cm}).
\end{align*}
Then $|\nabla\eta_{n}|\leq c\cdot|x-y|^{-1}$ holds uniformly in $n$, such that
\begin{align*}
	|\,\eta_{n}\hspace*{-0.02cm}\cdot\tau_{x}\phi_{\om}\,|\leq c\cdot|x-y|^{-1}
	\,,\quad\quad
	|\,\p_{j}\eta_{n}\cdot\tau_{x}\phi_{\om}\,|\leq c\cdot|x-y|^{-2}
	\,,\quad\quad
	|\,\eta_{n}\cdot\p_{j}(\tau_{x}\phi_{\om})\,|\leq c\cdot|x-y|^{-2}\,,
\end{align*}
where $\tau_{x}\phi_{\om}(y):=\phi_{\om}(x-y)$. Lebesgue's dominated 
convergence theorem shows
\begin{align*}
	\big(\hspace*{0.02cm}\phi_{\om}\ast\p_{j}F_{k}\hspace*{0.02cm}\big)(x)
		=\lim_{n\to\infty}
		 \scp{\tau_{x}\phi_{\om}}
		     {\ovl{\p_{j}(\hspace*{0.02cm}\eta_{n}F_{k})}\hspace*{0.03cm}}_{\lt}
		=\lim_{n\to\infty}
		 \scp{\tau_{x}\p_{j}\phi_{\om}}
		 	 {\ovl{\eta_{n}F_{k}}}_{\lt}
		=\big(\hspace*{0.02cm}\p_{j}
		 \phi_{\om}\ast F_{k}\hspace*{0.02cm}\big)(x)\,,
\end{align*}
which yields $\phi_{\om}\star\nabla\div F=\div F\star\nabla\phi_{\om}$ and
\begin{align*}
	-\phi_{\om}\star\rot F
		=\pthreevec{\,\phi_{\om}\ast\p_{3}F_{2}-\phi_{\om}\ast\p_{2}F_{3}\,}
				   {\,\phi_{\om}\ast\p_{1}F_{3}-\phi_{\om}\ast\p_{3}F_{1}\,}
				   {\,\phi_{\om}\ast\p_{2}F_{1}-\phi_{\om}\ast\p_{1}F_{2}\,}
	 	=\pthreevec{\,F_{2}\ast\p_{3}\phi_{\om}-F_{3}\ast\p_{2}\phi_{\om}\,}
				   {\,F_{3}\ast\p_{1}\phi_{\om}-F_{1}\ast\p_{3}\phi_{\om}\,}
				   {\,F_{1}\ast\p_{2}\phi_{\om}-F_{2}\ast\p_{1}\phi_{\om}\,}
		=:F\circledast \nabla\phi_{\om}\,.
\end{align*}
\begin{theo}\label{thm:lfa_repr-form}
	Let $0\neq\om\in\K\Subset\C_{+}$ and $\eps_{0},\mu_{0}\in\reals_{+}$. 
	Furthermore, let $1/2<s<3/2$, $t:=s-2$, and $(F,G)\in\Ds\times\Ds$.	
	Then for $(E,H):=\calL_{\Laz,\om}(F,G)$ the representation formulas
	\begin{align}
		E&=G\circledast\nabla\phi_{\om}
		   \,-\,i\om\mu_{0}\hspace*{0.03cm}\phi_{\om}\star F
		   -\frac{i}{\om\eps_{0}}\div F\star\nabla\phi_{\om},\\
		H&=-F\circledast \nabla\phi_{\om}
		   \,-\,i\om\eps_{0}\hspace*{0.03cm}\phi_{\om}\star G
		   -\frac{i}{\om\mu_{0}}\div G\star\nabla\phi_{\om}
	\end{align}
	hold in the sense of\;\;$\ltt$. Moreover, there exist $c>0$, such that for 
	all $\om\in\K\setminus\{0\}$ and all $(F,G)\in\Ds\times\Ds$
	\begin{align*}
		\norm{(E,H)}_{\Rt}
			\leq c\,\Big(\norm{(F,G)}_{\lts}
			     +\frac{1}{|\om|}\norm{(\div F,\div G)}_{\lts}\Big)\,.
	\end{align*}
\end{theo}
\begin{proof}
	Since $\cic\subset\Ds$ is dense, we choose a sequence 
	$\big((F_{n},G_{n})\big)_{n\in\mathbb{N}}\subset\cic\times\cic$ converging to
	$(F,G)$ and define 
	$(E_{n},H_{n})=\calL_{\Laz,\om}(F_{n},G_{n})\in\ltt\times\ltt$. Then 
	Remark \ref{rem:thp_sol-op} yields convergence of 
	$\big((E_{n},H_{n})\big)_{n\in\mathbb{N}}$ to $(E,H)\in\Rt\times\Rt$ 
	and as shown above, we may represent $(E_{n},H_{n})$ by
	\begin{align}
		E_{n}&=G_{n}\circledast\nabla\phi_{\om}
		   \,-\,i\om\mu_{0}\hspace*{0.03cm}\phi_{\om}\star F_{n}
		   -\frac{i}{\om\eps_{0}}\div F_{n}\star\nabla\phi_{\om},
		   \label{equ:lfa_repr-form_electric}\\
		H_{n}&=-F_{n}\circledast \nabla\phi_{\om}
		   \,-\,i\om\eps_{0}\hspace*{0.03cm}\phi_{\om}\star G_{n}
		   -\frac{i}{\om\mu_{0}}\div G_{n}\star\nabla\phi_{\om}\,.
		   \label{equ:lfa_repr-form_magnetic}
	\end{align}
	The involved convolution kernels essentially consist of 
	$\phi_{\om}$ and $\p_{j}\phi_{\om}$, which can be estimated by 
	\begin{align*}
		|\phi_{\om}|\,,\,|\p_{j}\phi_{\om}|
			\leq c\cdot\Big(\,|x-y|^{-1}+|x-y|^{-2}\,\Big)
		\,,\qquad (\,j=1,2,3\,)\,.
	\end{align*}
 	Moreover, from \cite[Lemma 1]{mcowen_behavior_1979} we obtain that 
 	integral operators with kernels of the form $|x-y|^{\alpha-\beta-3}$ 
 	map $\lgen{}{2}{\alpha}$ continuously to $\lgen{}{2}{\beta}$, 
 	if $-3/2<\alpha<\beta<3/2$. Hence, by choosing 
 	\[-3/2<t=s-2<\ttil:=s-1<s<3/2\,,\]
 	we have 
 	\begin{align*}
 		|x-y|^{-1}=|x-y|^{s-t-3}
 		\qquad\text{resp.}\qquad
 		|x-y|^{-2}=|x-y|^{s-\ttil-3}\,,
 	\end{align*}
 	and the right hand sides of \eqref{equ:lfa_repr-form_electric} and 
 	\eqref{equ:lfa_repr-form_magnetic} define bounded linear 
 	operators from $\lts$ to $\ltt$. Passing to the limit 
 	$n\To\infty$ in \eqref{equ:lfa_repr-form_electric},
 	\eqref{equ:lfa_repr-form_magnetic} we obtain the asserted representation 
 	formulas. By the continuity of the convolution operators we have the estimate 
 	 \begin{align*}
		\norm{(E,H)}_{\ltt}
			\leq c\,\Big(\norm{(F,G)}_{\lts}
				 +|\om|^{-1}\norm{(\div F,\div G)}_{\lts}\Big)	
	\end{align*}
 	which holds uniformly in $\om$. Finally the differential equation yields the 
 	asserted estimate.
\end{proof}
\noindent
A similar estimate also holds for radiating solutions in exterior weak Lipschitz 
domains.  
\begin{cor}\label{cor:lfa_a-priori-estimate}
	Let  $1/2<s<3/2$, $t:=s-2$, 
	and let $\eps,\mu$ be $\ka-\co-$decaying with order $\ka>2$, as well as let 
	$\K\Subset\C_{+}$. 
	Then there 
	exist $c,\da>0$ such that for all\;\,$0\neq\om\in\K$ and
	\[(F,G)\in\big(\Ds(\Om)\times\Ds(\Om)\big)\cap\gk{\scrM-\om}^{\perp}\]
	it holds
	\begin{align*}
		\norm{\calL_{\La,\om}(F,G)}_{\lttom}
			\leq c\,\Big(\,\normltsom{(F,G)}
		         +\frac{1}{|\om|}\normltsom{(\div F,\div G)}
		         +\norm{\calL_{\La,\om}(F,G)}_{\lt(\Om_{\da})}\Big)\,.
	\end{align*}
	Moreover, by the differential equation the 
	$\normlttom{\cdot}\hspace*{-0.1cm}-\hspace*{0.02cm}$norm on the 
	left hand side can be replaced by $\norm{\cdot}_{\Rtom}$.
\end{cor}
\begin{proof}
	Let $(E,H)=\calL_{\La,\om} (F,G)$ (\,which exists by 
	Remark \ref{rem:thp_sol-op}\,) and $\rtil>\rhat$ such that 
	$\eps,\mu\in\co(\cU_{\rtil})$. Then 
	\begin{align*}
		(\tilde{E},\tilde{H}):=\eta_{\rtil}(E,H)
		                       \in\Rsm{-\frac{1}{2}}\times\Rsm{-\frac{1}{2}}\,,	
	\end{align*}
	and as $(\,\M+i\om\La\,)\,(E,H)=(F,G)$ it holds 
	\begin{align}\label{equ:lfa_main-est_div}
		\hspace*{0.02cm}(\div\eps E,\div \mu H)
			=-\frac{i}{\om}\,(\div F,\div G)
			 \in\ltsom\times\ltsom	\,,
	\end{align}
	such that by Lemma \ref{lem:stp_reg-res} we even have
	\begin{align*}
		(E,H)\in\Hgen{}{1}{<-\frac{1}{2}}(\supp\eta_{\rtil})
		 \times
		 \Hgen{}{1}{<-\frac{1}{2}}(\supp\eta_{\rtil})\,,
		\qquad\text{especially}\qquad
		(\tilde{E},\tilde{H})\in\Hgen{}{1}{<-\frac{1}{2}}
		 \times\Hgen{}{1}{<-\frac{1}{2}}\,.
	\end{align*}
	Moreover, $(\tilde{E},\tilde{H})$ satisfies the radiation condition
	\begin{align*}
		\big(\,\Laz+\sqrt{\eps_{0}\mu_{0}}\;\Xi\;\big)\hspace*{0.03cm}
		(\tilde{E},\tilde{H})
		\in\ltbigom{-\frac{1}{2}}\times\ltbigom{-\frac{1}{2}}\,
	\end{align*} 
	and (as $\ka>2\geq s+1/2$) solves
	\begin{align}\label{equ:lfa_main-est_dgl-tsol}
		\big(\hspace*{0.03cm}\M+\,i\om&\Laz\hspace*{0.03cm}\big)\,
		 (\tilde{E},\tilde{H})
		   =\mathrm{C}_{\M,\eta_{\rtil}}(E,H)
		    -i\om\big(\La-\Laz\big)(\tilde{E},\tilde{H})+\eta_{\rtil}(F,G)
		   =:(\tilde{F},\tilde{G})\in\Ds\times\Ds\,,	
	\end{align} 
	where $\mathrm{C}_{A,B}:=AB-BA$. 
	We obtain $(\tilde{E},\tilde{H})=\calL_{\Laz,\om}(\tilde{F},\tilde{G})$ 
	and Theorem \ref{thm:lfa_repr-form} yields some $c>0$ such that
	\begin{align}\label{equ:lfa_main-est_fsest-tsol}
	  \big\|\,(\tilde{E},\tilde{H})\,\big\|_{\ltt}
	   \leq c\,\Big(\,\big\|\,(\tilde{F},\tilde{G})\,\big\|_{\lts}
		+\frac{1}{\om}\big\|\,(\div\tilde{F},\div\tilde{G})\,\big\|_{\lts}\,\Big)
	\end{align}
	independent of $\om$, $(\tilde{F},\tilde{G})$ or $(\tilde{E},\tilde{H})$. 
	Furthermore, \eqref{equ:lfa_main-est_div} and the differential equation 
	\eqref{equ:lfa_main-est_dgl-tsol} show
	\begin{alignat}{4}
		\div F&=i\om\div\eps E
		\,,\qquad & 
		       & 
		\qquad & 
		\div G&=i\om\div\mu H\,,
		\qquad \qquad &
		       &\text{ in }\Om\,,\label{equ:lfa_main-est_div-rhs}\\
		\div \tilde{F}&=i\om\eps_{0}\div\tilde{E}
		\,,\qquad & 
		       & 
		\qquad & 
		\div \tilde{G}&=i\om\mu_{0}\div\tilde{H} \,,
		\qquad \qquad &
		       &\text{ in }\rthree\,,\label{equ:lfa_main-est_div-trhs}
	\end{alignat}
	such that combining \eqref{equ:lfa_main-est_fsest-tsol} and 
	\eqref{equ:lfa_main-est_div-trhs} it holds
	\begin{align*}
	  \hspace*{-0.8cm}\big\|\,(E,H)\,\big\|_{\lttom}
	  &\leq c\,\Big(\,\big\|\,(E,H)\,\big\|_{\lt(\Om_{2\rtil})}
	   +\big\|\,(\tilde{E},\tilde{H})\,\big\|_{\ltt(\cU_{\rtil})}\,\Big)\\
	  &\leq c\,\Big(\,\big\|\,(E,H)\,\big\|_{\lt(\Om_{2\rtil})}
	   +\big\|\,(\tilde{F},\tilde{G})\,\big\|_{\lts}
	   +\frac{1}{|\om|}\big\|\,(\div\tilde{F},\div\tilde{G})\,\big\|_{\lts}
	   \,\Big)\\
	  &\leq c\,\Big(\,\big\|\,(E,H)\,\big\|_{\ltsmka(\Om)}
	   +\big\|\,(F,G)\,\big\|_{\ltsom}
	   +\big\|\,(\div\eps_{0}\tilde{E},\div\mu_{0}\tilde{H})\,\big\|_{\lts}\,\Big)\,.
	\end{align*}
	With \eqref{equ:lfa_main-est_div-rhs} the last term on the right hand 
	side can be estimated by
	\begin{align*}
	  &\hspace*{-0.8cm}
	  \big\|\,(\div\eps_{0}\tilde{E},\div\mu_{0}\tilde{H})\,\big\|_{\lts}\\
	  &\leq c\,\Big(\,\norm{(E,H)}_{\lt(\Om_{2\rtil})}
	   +\norm{(\div \eps_{0}E,\div \mu_{0}H)}_{\lts(\supp\eta_{\rtil})}\,\Big)\\
	  &\leq c\,\Big(\,\norm{(E,H)}_{\lt(\Om_{2\rtil})}
	   +\norm{(\div\eps E,\div\mu H)}_{\lts(\supp\eta_{\rtil})}
	   +\norm{(\div\hat{\eps}E,\div\hat{\mu}H)}_{\lts(\supp\eta_{\rtil})}\,\Big)\\
	  &\leq c\,\Big(\,\norm{(E,H)}_{\lt(\Om_{2\rtil})}
	   +\frac{1}{|\om|}\norm{(\div F,\div G)}_{\ltsom}
	   +\norm{(E,H)}_{\hgen{}{1}{s-\ka-1}(\supp\eta_{\rtil})}\,\Big)\,.
	\end{align*}
	We end up with
	\begin{align*}
	  \big\|\,(E,H)\,\big\|_{\lttom}
	  &\leq c\,\Big(\,\norm{(E,H)}_{\ltsmka(\Om)}
	   +\norm{(E,H)}_{\hgen{}{1}{s-\ka-1}(\supp\eta_{\rtil})}\\
	  &\qquad\qquad+\norm{(F,G)}_{\ltsom}
	   +\frac{1}{|\om|}\norm{(\div F,\div G)}_{\ltsom}\,\Big)
	\end{align*}
	and the estimate from Lemma \ref{lem:stp_reg-res} as well as the 
	differential equation together with \eqref{equ:lfa_main-est_div-rhs} 
	yield
	\begin{align*}
	  \norm{(E,H)}_{\lttom}
	  \leq c\,\Big(\,\norm{(E,H)}_{\ltsmka(\Om)}+\norm{(F,G)}_{\ltsom}
	   +\frac{1}{|\om|}\norm{(\div F,\div G)}_{\ltsom}\,\Big)\,.
	\end{align*}
	Finally, as $\ka>2$ the assertion follows by 
	\begin{align*}
		\norm{(E,H)}_{\ltsmka(\Om)}^2
			 \leq\norm{(E,H)}_{\lt(\Om_{\da})}^2
			  +\big(1+\da^2\big)^{2-\ka}
			  \cdot\norm{(E,H)}_{\lttom}^2\,,
	\end{align*}
	choosing $\da>\rhat$ big enough.
\end{proof}
\begin{theo}\label{thm:lfa_spectrum_central-est}
	Let $1/2<s<3/2$, $t:=s-2$, and let $\eps,\mu$ be $\ka-\co-$decaying with 
	order $\ka>2$, and let
	\begin{align*}
		\BGatom=\{\scrB_{1,1},\ldots,\scrB_{1,d_{1\hspace*{-0.025cm},
		\hspace*{-0.02cm}2}}\}\subset\Rztom
		\qquad\text{resp.}\qquad
		\BGanom=\{\scrB_{2,1},\ldots,
				  \scrB_{2,d_{2\hspace*{-0.025cm},\hspace*{-0.02cm}1}}\}
				  \subset\Rznom
	\end{align*} 
	be the sets from Theorem \ref{thm:stp_dir-neu-char}. Then:\\[-8pt]
	\begin{enumerate}[leftmargin=0.85cm,label=$(\roman*)$,itemsep=7pt]
		\item $\gs$ has no accumulation point at zero. In particular, there exists 
			  some $\tilde{\om}>0$ such that 
			  \begin{align*}
			  	\gs\cap\C_{+,\tilde{\om}}=\emptyset
			  	\qquad\text{with}\qquad
			  	\C_{+,\tilde{\om}}
			  	 :=\setb{\,\om\in\C_{+}\,}{\,|\om|\leq\tilde{\om}\,}\,.
			  \end{align*}
		\item $\calL_{\La,\om}$ is well defined on the whole of 
			  $\ltbig{\frac{1}{2}}(\Om)\times\ltbig{\frac{1}{2}}(\Om)$
			  for all $\om\in\C_{+,\tilde{\om}}\sm\{0\}$.
		\item 
			  There exists a constant $c>0$ such that 
			  \begin{align*}
			  	\normlttom{\calL_{\La,\om}(F,G)}
			  		&\leq c\,\Big(\normltsom{(F,G)}
			  		             +\frac{1}{|\om|}\normltsom{(\div F,\div G)}\\
			  		&\qquad
			  		 +\frac{1}{|\om|}\sum_{\ell=1,\ldots,
			  		  d_{1\hspace*{-0.025cm},\hspace*{-0.02cm}2}}
			  		  |\scp{F}{\scrB_{1,\ell}}_{\ltom}|
			  	     +\frac{1}{|\om|}\sum_{\ell=1,\ldots,
			  	      d_{2\hspace*{-0.025cm},\hspace*{-0.02cm}1}}
			  		  |\scp{G}{\scrB_{2,\ell}}_{\ltom}|          
			  		\Big)
			  \end{align*}
			  holds for all $\om\in\C_{+,\tilde{\om}}\sm\{0\}$ and 
			  $(F,G)\in\Dsnom\times\Dstom$. Using the differential equation, the 
			  $\normlttom{\cdot}\hspace*{-0.1cm}-$norm on the left hand side may 
			  be replaced by the natural norm in 
			  \begin{align*}
			  	\big(\,\Rttom\cap\eps^{-1}\Dtnom\,\big)
			  	\times
			    \big(\,\Rtnom\cap\mu^{-1}\Dttom\,\big)\,.\\[-8pt]
			  \end{align*}
	\end{enumerate}
\end{theo}
\begin{proof}
	Assuming that zero is an accumulation point of $\gs$ there exist a 
	sequence $(\om_n)_{n\in\mathbb{N}}\subset\reals\sm\{0\}$ 
	(\,cf. Theorem \ref{thm:thp_fredh-alt} (iii)\,) tending to zero 
	and a sequence $\big((E_{n},H_{n})\big)_{n\in\mathbb{N}}$ with 
	$(E_{n},H_{n})\in\gk{\scrM-\om_{n}}$ and 
	\[\normlttom{(E_{n},H_{n})}=1\quad\;\,\text{for some}\quad-3/2<t<-1/2\,.\]
	Using the differential equation we obtain 
	$(E_{n},H_{n})\in\big(\,\Rttom\cap\eps^{-1}\zdtnom\,\big)
		\times
	  	\big(\,\Rtnom\cap\mu^{-1}\zdttom\,\big)$ with
	\begin{align*}
		\normlttom{(\rot E_{n},\rot H_{n})}
		\leq c\cdot|\om_{n}|\cdot\normlttom{(E_{n},H_{n})}
		\xrightarrow{\;n\rightarrow\infty\;}0\,.
	\end{align*}
	Consequently $\big((E_{n},H_{n})\big)_{n\in\mathbb{N}}$ is bounded in 
	\begin{align*}
		\big(\,\Rttom\cap\eps^{-1}\Dtnom\,\big)
		\times
	  	\big(\,\Rtnom\cap\mu^{-1}\Dttom\,\big)\,.
	\end{align*}
	Thus Weck's local selection theorem yields a subsequence 
	$\big((E_{\pi(n)},H_{\pi(n)})\big)_{n\in\mathbb{N}}$ 
	converging in $\lgen{}{2}{\tilde{t}}(\Om)\times\lgen{}{2}{\tilde{t}}(\Om)$
	for all $\ttil<t$. In particular, as $t>-3/2$ we may assume 
	$t>\tilde{t}\geq -3/2$. Then 	
	$\big((E_{\pi(n)},H_{\pi(n)})\big)_{n\in\mathbb{N}}$ converges in 
	$\big(\,\Rttiltom\cap\eps^{-1}\zdttiln(\Om)\,\big)\times
	\big(\,\Rttilnom\cap\mu^{-1}\zdttilt(\Om)\,\big)$
	to some\\[-10pt]
	\begin{align*}
		(E,H)\in\harmgen{\eps}{\tilde{t}}{\Gat}{\Gan}(\Om)
		     \times\harmgen{\mu}{\tilde{t}}{\Gan}{\Gat}(\Om)
		     \overset{\eqref{equ:stp_dir-neu-fields_equal}}{=}
		     \vHarmtn{\eps}(\Om)\times\vHarmnt{\mu}(\Om)\,.
	\end{align*}
	In addition, the differential equation together with 
	\eqref{equ:stp_dir-neu-gen_int} yields
	\begin{align*}
		(E_{\pi(n)},H_{\pi(n)})
		\in\BGatom^{\perp_{\eps}}\times\BGanom^{\perp_{\mu}}
		\quad\Longrightarrow\quad
		(E,H)\in\BGatom^{\perp_{\eps}}\times\BGanom^{\perp_{\mu}}\,.	
	\end{align*}
	Therefore by Theorem \ref{thm:stp_dir-neu-char}
	\begin{align*}
		(E,H)
		\in\big(\,\vHarmtn{\eps}(\Om)\cap\BGatom^{\perp_{\eps}}\,\big)
		\times\big(\,\vHarmnt{\mu}(\Om)\cap\BGanom^{\perp_{\mu}}\,\big)
		=\{0\}\times\{0\}\,.
	\end{align*}
	Finally Corollary \ref{cor:lfa_a-priori-estimate} 
	yields constants $c,\da>0$ independent of $n$ such that 
	\begin{align*}
		1=\normlttom{(E_{\pi(n)},H_{\pi(n)})}
		 \leq c\cdot\norm{(E_{\pi(n)},H_{\pi(n)})}_{\lt(\Omda)}
		\xrightarrow{n\to\infty}0\,,
	\end{align*}
	a contradiction which proves $(i)$\;resp.\;$(ii)$. In order to prove $(iii)$, 
	we assume that the asserted estimate is wrong. Then we obtain sequences 
	$(\om_{n})_{n\in\mathbb{N}}\subset\C_{+,\tilde{\om}}\sm\{0\}$ tending to zero and 
	\begin{align*}
		\big((F_{n},G_{n})\big)_{n\in\mathbb{N}}\subset\Dsnom\times\Dstom
		\qquad\text{with}\qquad 
		\normlttom{\calL_{\La,\om_{n}}(F_{n},G_{n})}=1
	\end{align*}
	such that 
	\begin{align*}
		\normltsom{(F_{n},G_{n})}
		 \xrightarrow{\;n\To\infty\;}0
		\,,\qquad\qquad
		|\om_{n}|^{-1}\cdot\normltsom{(\div F_{n},\div G_{n})}
		 \xrightarrow{\;n\To\infty\;}0\,,
	\end{align*}
	and 
	\begin{align}
		&|\om_{n}|^{-1}\cdot|\scp{F_{n}}{\scrB_{1,\ell}}_{\ltom}|
		 \xrightarrow{\;n\To\infty\;} 0
		 \,,\qquad\qquad \ell=1,\ldots,d_{1,2}\,,
		 \label{equ:lfa_acc-points_orth-const}\\
		&|\om_{n}|^{-1}\cdot|\scp{G_{n}}{\scrB_{2,\ell}}_{\ltom}|
		 \xrightarrow{\;n\To\infty\;}0
		 \,,\qquad\qquad \ell=1,\ldots,d_{2,1}\,.
		 \label{equ:lfa_acc-points_orth-const-2}
	\end{align}
	As above, the differential equation shows 
	$\big((E_{n},H_{n})\big)_{n\in\mathbb{N}}$ with 
	$(E_{n},H_{n}):=\calL_{\La,\om_{n}}(F_{n},G_{n})$ is bounded in 
	\begin{align*}
		\big(\,\Rttom\cap\eps^{-1}\Dtnom\,\big)
		\times
	  	\big(\,\Rtnom\cap\mu^{-1}\Dttom\,\big)
	\end{align*}
	and again Weck's local selection theorem provides a subsequence
	$\big((E_{\pi(n)},H_{\pi(n)})\big)_{n\in\mathbb{N}}$ converging in 
	\begin{align*}
		\big(\,\Rttiltom\cap\eps^{-1}\Dttilnom\,\big)
		\times
	  	\big(\,\Rttilnom\cap\mu^{-1}\Dttiltom\,\big)
	\end{align*}
	for all $-3/2\leq\tilde{t}<t$. We obtain
	\begin{align*}
		(E,H):=\lim_{n\to\infty}	(E_{\pi(n)},H_{\pi(n)})
		\in\harmgen{\eps}{\tilde{t}}{\Gat}{\Gan}(\Om)
		\times\harmgen{\mu}{\tilde{t}}{\Gan}{\Gat}(\Om)
		\overset{\eqref{equ:stp_dir-neu-fields_equal}}{=}
		\vHarmtn{\eps}(\Om)\times\vHarmnt{\mu}(\Om)\,.
	\end{align*}
	Moreover, by \eqref{equ:lfa_acc-points_orth-const} we compute for 
	$\ell=1,\ldots,d_{1,2}$
	\begin{align*}
		0\xleftarrow{\;n\To\infty\;}\;
	 	&|\om_{n}|^{-1}\cdot|\scp{F_{n}}{\scrB_{1,\ell}}_{\ltom}|\\
	 	&\;=|\om_{n}|^{-1}\cdot
	 	  |\underbrace{\scp{\rot H_{n}}{\scrB_{1,\ell}}_{\ltom}}_{0}
	 	  +i\hspace*{0.02cm}\om_{n}\scp{\eps E_{n}}{\scrB_{1,\ell}}_{\ltom}|
	      \xrightarrow{\;n\To\infty\;}
	      |\scp{\eps E}{\scrB_{1,\ell}}_{\ltom}|\,,
	\end{align*}
	hence $E\in\BGatom^{\perp_{\eps}}$ and with 
	\eqref{equ:lfa_acc-points_orth-const-2} analogously 
	$H\in\BGanom^{\perp_{\mu}}$. Thus $(E,H)$ must vanish and again Corollary 
	\ref{cor:lfa_a-priori-estimate} yields constants $c,\da>0$ independent of 
	$n$ such that  
	\begin{align*}
		1&=\normlttom{(E_{n},H_{n})}\\
			&\qquad\leq c\,\Big(\,\normltsom{(F_{n},G_{n})}
			 +|\om_{n}|^{-1}\normltsom{(\div F_{n},\div G_{n})}
		     +\norm{(E_{n},H_{n})}_{\lt(\Om_{\da})}\Big)
		      \xrightarrow{\;n\To\infty\;}0\,,
	\end{align*}
	a contradiction.
\end{proof}
\noindent
We are ready to prove our main result:
\begin{theo}
	Let $\eps,\mu$ be 
	$\ka-\co-$decaying with order $\ka>2$, $1/2<s<3/2$, $t:=s-2$, and 
	let $\tilde{\om}$ be the radius from Theorem \ref{thm:lfa_spectrum_central-est}. 
	Then for $(\om_{n})_{n\in\mathbb{N}}\subset\C_{+,\tilde{\om}}\sm\{0\}$ tending 
	to zero and 
	\begin{align*}
		\big((F_{n},G_{n})\big)_{n\in\N}\subset\Dsnom\times\Dstom
	\end{align*}
	such that
	\begin{alignat*}{2}
		(F_{n},G_{n})
		  &\;\xrightarrow{\;n\To\infty\;}\;(F,G)\qquad & 
		&\text{in}\qquad\ltsom\times\ltsom\,,\\
		 -i\om_{n}^{-1}(\div F_{n},\div G_{n})
		  &\;\xrightarrow{\;n\To\infty\;}\;(f,g)\qquad & 
		&\text{in}\qquad\ltsom\times\ltsom\,,\\
		 -i\om_{n}^{-1}\scpltom{F_{n}}{\scrB_{1,\ell}}
		  &\;\xrightarrow{\;n\To\infty\;}\;\zeta_{\ell}\qquad & 
		&\text{in}\qquad\C\,,\quad\quad \ell=1,\ldots,d_{1,2}\,,\\
		 -i\om_{n}^{-1}\scpltom{G_{n}}{\scrB_{2,\ell}}
		  &\;\xrightarrow{\;n\To\infty\;}\;\theta_{\ell}\qquad & 
		&\text{in}\qquad\C\,,\quad\quad \ell=1,\ldots,d_{2,1}\,,
	\end{alignat*}
	the sequence $\big((E_{n},H_{n})\big)_{n\in\N}
	:=\big(\calL_{\La,\om_{n}}(F_{n},G_{n})\big)_{n\in\N}$ of radiating solutions
	converges for all $\ttil<t$ in 
	\begin{align*}
		\big(\,\Rttiltom\cap\eps^{-1}\Dttilnom\,\big)
		\times
		\big(\,\Rttilnom\cap\mu^{-1}\Dttiltom\,\big)
	\end{align*}
	to the static solutions $(E,H)\in\big(\,\Rmotom\cap\eps^{-1}\Dmonom\,\big)
	\times\big(\,\Rmonom\cap\mu^{-1}\Dmotom\,\big)$ of 
	\begin{alignat*}{3}
		\rot E&=G
		\,,\qquad\qquad &
		\div\eps E&=f
		\,,\qquad\qquad &
		\scp{E}{\scrB_{1,\ell}}_{\lteps(\Om)}&=\zeta_{\ell}
		\quad\quad(\,
		\ell=1,\ldots,d_{1,2}\,)\,,\\
		\rot H&=F
		\,,\qquad\qquad &
		\div\mu H&=g
		\,,\qquad\qquad &
		\scp{H}{\scrB_{2,\ell}}_{\ltmu(\Om)}&=\theta_{\ell}
		\quad\quad(\,
		\ell=1,\ldots,d_{2,1}\,)\,.
	\end{alignat*}
\end{theo}
\begin{proof}
	By Theorem \ref{thm:lfa_spectrum_central-est} $(iii)$ the sequence 
	$\big((E_{n},H_{n})\big)_{n\in\N}$ is bounded in 
	\begin{align*}
		\big(\,\Rttom\cap\eps^{-1}\Dtnom\,\big)
		\times
		\big(\,\Rtnom\cap\mu^{-1}\Dttom\,\big)
	\end{align*}
	and the differential equation yields 
	\begin{align*}
		M(E_{n},H_{n})=(F_{n},G_{n})-i\om_{n}\La(E_{n},H_{n})
		\,,\qquad\qquad
		(\div\eps E_{n},\div\mu H_{n})&=-\frac{i}{\om_{n}}(\div F_{n},\div G_{n})\,,
	\end{align*}
	such that by assumption
	\begin{alignat*}{3}
		(\rot E_{n},\rot H_{n})
		&\xrightarrow{\;n\To\infty\;}(F,G)
		\qquad &&\text{in}\qquad &&\lttom\times\lttom\,,\\
		(\div \eps E_{n},\div \mu H_{n})
		&\xrightarrow{\;n\To\infty\;}(f,g)
		\qquad &&\text{in}\qquad &&\ltsom\times\ltsom\,.
	\end{alignat*}
	Moreover, for $\ell=1,\ldots,d_{1,2}$ we compute by 
	\eqref{equ:stp_dir-neu-gen_int}
	\begin{align*}
		\scp{E_{n}}{\scrB_{1,\ell}}_{\ltepsom}
		 =-\frac{i}{\om_{n}}
		  \underbrace{\scpltom{\rot H_{n}}{\scrB_{1,\ell}}}_{=\,0}
		  -\frac{i}{\om_{n}}\scpltom{F_{n}}{\scrB_{1,\ell}}
		\xrightarrow{\;n\To\infty\;}\zeta_{\ell}
	\end{align*}
	and analogously $\scp{H_{n}}{\scrB_{2,\ell}}_{\ltmuom}
	\xrightarrow{\;n\To\infty\;}\theta_{\ell}$ for 
	$\ell=1,\ldots,d_{2,1}$.
	By Weck's local selection theorem we may extract a subsequence 
	$\big((E_{\pi(n)},H_{\pi(n)})\big)_{n\in\N}$ with
	\begin{align*}
		(E_{\pi(n)},H_{\pi(n)})
		\xrightarrow{\;n\To\infty\;}:(\tilde{E},\tilde{H})
		\qquad\text{in}\qquad
		\ltttilom\times\ltttilom
	\end{align*}
	for all $-3/2<\ttil<t$. Then 
	\begin{align*}
		(\tilde{E},\tilde{H})
		\in\big(\,
		   \rbigt{-\frac{3}{2}}(\Om)\cap\eps^{-1}\dbign{-\frac{3}{2}}(\Om)\,
		   \big)
		\times\big(\,
		   \rbign{-\frac{3}{2}}(\Om)\cap\mu^{-1}\dbigt{-\frac{3}{2}}(\Om)\,
		   \big)
	\end{align*}
	and $(\tilde{E},\tilde{H})$ solves the electro-magneto static system 
	\begin{alignat*}{3}
		\rot \tilde{E}&=G
		\,,\qquad\qquad &
		\div\eps \tilde{E}&=f
		\,,\qquad\qquad &
		\scp{\tilde{E}}{\scrB_{1,\ell}}_{\lteps(\Om)}&=\zeta_{\ell}
		\quad\quad(\,
		\ell=1,\ldots,d_{1,2}\,)\,,\\
		\rot \tilde{H}&=F
		\,,\qquad\qquad &
		\div\mu \tilde{H}&=g
		\,,\qquad\qquad &
		\scp{\tilde{H}}{\scrB_{2,\ell}}_{\ltmu(\Om)}&=\theta_{\ell}
		\quad\quad(\,
		\ell=1,\ldots,d_{2,1}\,)\,.
	\end{alignat*}
	Finally, the difference $(e,h):=(E,H)-(\tilde{E},\tilde{H})$ satisfies
	\begin{align*}
		(e,h)\in\big(\,
		\harmgen{\eps}{>-\frac{3}{2}}{\Gat}{\Gan}(\Om)\cap\BGatom^{\perp_{\eps}}
		\,\big)
		\times\big(\,
		\harmgen{\mu}{>-\frac{3}{2}}{\Gan}{\Gat}(\Om)\cap\BGanom^{\perp_{\mu}}
		\,\big)\,.
	\end{align*}
	Hence, by \eqref{equ:stp_dir-neu-fields_equal} and Theorem 
	\ref{thm:stp_dir-neu-char} we have $(E,H)=(\tilde{E},\tilde{H})$ and due 
	to the uniqueness of the limit $(E,H)$ even the whole sequence 
	$\big((E_{n},H_{n})\big)_{n\in\N}$ must converge to $(E,H)$ in 
	$\ltsm{t}(\Om)\times\ltsm{t}(\Om)$.
\end{proof}
%
%
%
\bibliographystyle{plain} 
\bibliography{static-solution_low-frequency-biblio-v04-siam.bib}
%
%
%
%
\appendix 
%
%
\section{Proof of Theorem \ref{thm:stp_dir-neu-char}}
\label{sec:app_proof-dir-neu-replace-rot}
%
%
\noindent
Without loss of generality we concentrate on the construction of $\BGatom$ for 
$\ga=\mathbbm{1}$. As mentioned, the idea is to construct $\BGatom$ using a basis 
$\frakB(\Omrhat)$ of $\vHarmv{}{\Ga_{1,\rhat}}{\Gan}(\Omrhat)$, 
$\Ga_{1,\rhat}:=\Gat\cup\Sp_{\rhat}$. More precisely, we define\\

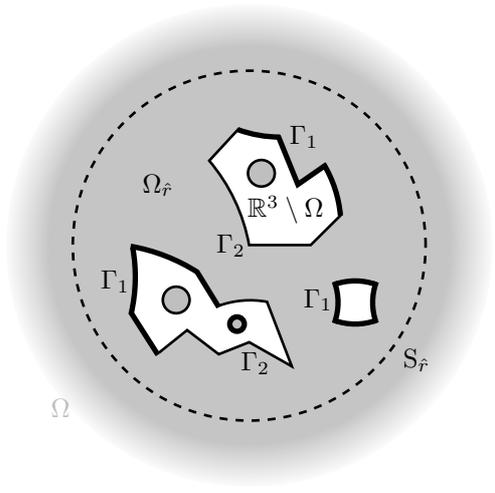
\begin{minipage}{0.4\textwidth}
\centering
\begin{tikzpicture}[scale=0.8]
	\tikzset{
  		ring shading/.code args={from #1 at #2 to #3 at #4}{
    	\def\colin{#1}
    	\def\radin{#2}
    	\def\colout{#3}
    	\def\radout{#4}
    	\pgfmathsetmacro{\proportion}{\radin/\radout}
    	\pgfmathsetmacro{\outer}{0.8818cm}
    	\pgfmathsetmacro{\inner}{0.8818cm*\proportion}
    	\pgfmathsetmacro{\innerlow}{\inner-0.01pt}
    	\pgfdeclareradialshading{ring}{\pgfpoint{0cm}{0cm}}%
    	{
    	  color(0pt)=(white);
    	  color(\innerlow)=(white);
    	  color(\inner)=(#1);
    	  color(\outer)=(#3)
    	}
    	\pgfkeysalso{/tikz/shading=ring}
  		},
	}
	\begin{scope} 
	    \shade[even odd rule, 
	    ring shading={from lightgray at 3.3 to white at 4.1}]
  		(0,0) circle (4cm) circle (3.3cm);
		\fill[lightgray,opacity=0.9] (0,0) circle (3.3cm);
		\draw[dashed,line width=1pt] (0,0) circle (2.9cm);
		\filldraw[white] (0,0)--(1,0)--(1.5,0.5)
		                     --(1.5,0.5) arc(5:29:2)--(1.24,1.32)
		                     --(0.8,1)--(0.5,1.8)
		                     --(0.505,1.8) arc(270:249.76:2)--(-0.19,1.92)
		                     --(-0.65,1.405)--(-0.65,1.405) arc (40:10:3);
		\draw[line width=1pt](-0.017,0.01)--(1.01,0.01)-- (1.52,0.52);
		\draw[line width=2pt] (1.5,0.5) arc(5:31:2);
		\draw (1.25,1.34) -- (1.25,1.378);
		\draw[line width=2pt] (1.265,1.34)--(0.8,1.015)--(0.474,1.822);
		\draw[line width=2pt] (0.505,1.8) arc(270:249.7:2);
		\draw[line width=1pt] (-0.668,1.396)--(-0.165,1.945);
		\draw[line width=1pt] (0,0) arc (10:40.2:3);
		\filldraw[lightgray] (0.2,1.2) circle (0.22cm);
		\draw[line width=1pt] (0.2,1.2) circle (0.22cm);
		\filldraw[white] (-1.95,-1.14) arc (-12.8:10:2.75)
							 --(-1.91,0) arc (80:57.2:2.7)
							 --(-0.5,-1)--(-0.85,-0.43)
							 --(-0.5,-1) arc(110:79:1.5)
							 --(0.3,-0.94)--(0.7,-2)--(0,-1.6)
							 --(-0.5,-1.8)--(-1,-1.4)--(-1.5,-1.8)
							 --(-1.96,-1.1);
		\draw[line width=2pt] (-1.95,-0.01) arc (80:57.2:3);
		\draw[line width=2pt] (-1.914,-0.04) arc (10:-13:2.75);
		\draw[line width=2pt] (-0.499,-1.02)--(-0.86,-0.43);
		\draw[line width=1pt] (-0.51,-1) arc(110.5:79:1.5);
		\draw[line width=1pt] (0.288,-0.919)--(0.7,-2)--(0,-1.6)--
							  (-0.5,-1.8)--(-1.02,-1.4)--(-1.48,-1.772);
		\draw[line width=2pt] (-1.5,-1.81)--(-1.95,-1.1);
		\filldraw[lightgray] (-1.2,-0.9) circle (0.22cm);
		\draw[line width=1pt] (-1.2,-0.9) circle (0.22cm);
		\filldraw[lightgray] (-0.2,-1.3) circle (0.12cm);
		\draw[line width=2pt] (-0.2,-1.3) circle (0.12cm);
		\filldraw[white] (1.4,-0.635) arc (110:70:0.9)
							 --(2.06,-0.64) arc (160:200:0.9)
							 --(1.97,-1.25) arc (288:255:0.9);
							 --(1.4,-0.6) arc (20:-20:0.9);	
		\draw[line width=2pt] (1.4,-0.644) arc (110:70:1);
		\draw[line width=2pt] (1.4,-0.598) arc (20.1:-20.1:1);
		\draw[line width=2pt] (1.4,-1.239) arc (250:290:1);
		\draw[line width=2pt] (2.09,-0.6) arc (160:200:1);
		\draw (-1.5,1) node {$\Omrhat$};
		\draw (0.6,0.6) node {$\rthree\sm\Om$};
		\draw (2.75,-1.9) node {$\Sp_{\rhat}$};
		\draw[lightgray] (-3.1,-2.7) node {$\Om$};
		\draw (0.9,1.8) node {$\Gat$};
		\draw (-2.2,-0.6) node {$\Gat$};
		\draw (1.13,-0.9) node {$\Gat$};
		\draw (0.1,-1.95) node {$\Gan$};
		\draw (-0.3,0) node {$\Gan$};
	\end{scope}
\end{tikzpicture}
\end{minipage}
\begin{minipage}{0.6\textwidth}
	\begin{align*}
		\BGatom:=\setb{\scrE_{\Om}(\B)}{\B\in\frakB(\Omrhat)}
				 \subset\zrztom\,,	
	\end{align*}
	where $\map{\scrE_{\Om}}{\lt(\Omrhat)}{\ltom}$ extends functions\;resp.\;fields 
	defined on $\Omrhat$ by zero to $\Om$, and show the following: \\[-8pt]
	\begin{enumerate}[leftmargin=1.75cm, itemsep=4pt, 
	label=\textbf{Step \arabic*:}]
		\item Choosing a basis $\frakB(\Omrhat)$ of 
			  $\vHarmv{}{\Ga_{1,\rhat}}{\Gan}(\Omrhat)$, extending the elements in 
			  $\frakB(\Omrhat)$ by zero to $\Om$ and projecting them onto
			  $\vHarmtn{}(\Om)$, we obtain a linearly independent subset of 
			  $\vHarmtn{}(\Om)$, 
		\item Choosing a basis $\frakB(\Om)$ of $\vHarmtn{}(\Om)$, restricting the 
			  elements in $\frakB(\Om)$ to $\Omrhat$ and projecting them onto 
			  $\vHarmv{}{\Ga_{1,\rhat}}{\Gan}(\Omrhat)$, we obtain a linearly 
			  independent subset of $\vHarmv{}{\Ga_{1,\rhat}}{\Gan}(\Omrhat)$.
			  \\[-8pt]
	\end{enumerate} 
	Then, Step 1 and Step 2 already imply 
	(\,cf. \eqref{equ:dnf_transf-ind}\,)
	\begin{align*}
		|\BGatom|=\dim\vHarmv{}{\Ga_{1,\rhat}}{\Gan}(\Omrhat)
		  		 =\dim\vHarmtn{}(\Om)=d_{1,2}<\infty\,.\\[-10pt]
	\end{align*}
\end{minipage}

Moreover, by Step 1 the projections of the elements in $\BGatom$ along 
$\ovl{\nabla\Hoztom}$ are linearly independent and thus form a basis of 
the Dirichlet-Neumann fields $\vHarmtn{}(\Om)$. Hence, it just remains to show:
\begin{align*}
	\text{\bf Step 3:}\quad\vHarmtn{}(\Om)\cap\BGatom^{\perp}=\{0\}
\end{align*} 	
\begin{lem}[Step 1]\label{lem:app_proof-dir-neu_step-1}
	Let $\map{\pi}{\zrztom}{\vHarmtn{}(\Om)}$ be the orthogonal projection 
	given by 
	\begin{align}\label{equ:app_proof-dir-neu_zrztom-decomp}
		\zrztom=\nabla\homotom\oplus\vHarmtn{}(\Om)
	\end{align}
	from Remark \ref{rem:stp_toolbox-reproduced}. Then the composition
	\begin{align*}
		 \map{\pi\circ\mathscr{E}_{\Om}}
	     {\vHarmv{}{\Ga_{1,\rhat}}{\Gan}(\Omrhat)}
	     {\vHarmtn{}(\Om)}
	\end{align*}
	is injective.
\end{lem}
\begin{proof}
	Let $H\in\vHarmv{}{\Ga_{1,\rhat}}{\Gan}(\Omrhat)$. Then 
	$\mathscr{E}_{\Om}(H)\in\zrztom$ and  
	with \eqref{equ:app_proof-dir-neu_zrztom-decomp} we can decompose 
	\begin{align*}
		\mathscr{E}_{\Om}(H)
			=\nabla w+\theta\in\nabla\homotom\oplus\vHarmtn{}(\Om)\,.	
	\end{align*}
 	To show injectivity we assume $\theta=0$. Then $\nabla w=\mathscr{E}_{\Om}(H)=0$ 
 	in $\cU_{\rhat}$. Thus $w$ is constant in $\cU_{\rhat}$ and as $w\in\homotom$ it 
 	has to vanish in $\cU_{\rhat}$, hence 
 	$w\in\Hgen{}{1}{\Ga_{1,\text{$\rhat$}}}(\Omrhat)$. By partial integration we 
 	conclude
	\begin{align*}
		\norm{H}_{\lt(\Omrhat)}^2
			=\scp{H}{\nabla w}_{\lt(\Omrhat)}
			=-\scp{\div H}{w}_{\lt(\Omrhat)}=0\,.
	\end{align*}
\end{proof}
\begin{lem}[Step 2]\label{lem:app_proof-dir-neu_step-2}
	Let $\map{\pi}{\zdzn(\Omrhat)}{\vHarmv{}{\Ga_{1,\rhat}}{\Gan}(\Omrhat)}$ 
	be the orthogonal projection given by 
	\begin{align}\label{equ:app_proof-dir-neu_zdznom-decomp}
		\zdzn(\Omrhat)=\rot\Rzn(\Omrhat)
		\oplus\vHarmv{}{\Ga_{1,\rhat}}{\Gan}(\Omrhat)
	\end{align}
	from Lemma \ref{lem:stp_fa-toolbox} (iii). Moreover, let 
	$\map{\scrR_{\Omrhat}}{\ltom}{\lt(\Omrhat)}$ be the operator restricting 
	functions\;resp.\;fields on $\Om$ to $\Omrhat$. Then
	\begin{align*}
		\map{\pi\circ\scrR_{\Omrhat}}
		 	{\vHarmtn{}(\Om)}
		 	{\vHarmv{}{\Ga_{1,\rhat}}{\Gan}(\Omrhat)}
	\end{align*}
	is injective.
\end{lem}
\begin{proof}
	Let $H\in\vHarmtn{}(\Om)$. By 
	\eqref{equ:app_proof-dir-neu_zdznom-decomp}, the restriction
	$\scrR_{\Omrhat}(H)\in\zdzn(\Omrhat)$ can be decomposed into   
	\begin{align*}
		\mathscr{R}_{\Omrhat}(H)
			=\rot E+\theta
			 \in\rot\Rzn(\Omrhat)
			 \oplus\vHarmv{}{\Ga_{1,\rhat}}{\Gan}(\Omrhat)\,.
	\end{align*}
	To show injectivity we assume $\theta=0$. In $\Omrhat$ we have 
	\begin{align}\label{equ:app_proof-dir-neu_step-1-bd}
		H=\mathscr{R}_{\Omrhat}(H)=\rot E
		\quad\text{with}\quad
		E\in\Rzn(\Omrhat)\,.	
	\end{align}
	Furthermore, $H\in\zr(\cU_{\rhat})$ and as the \emph{Neumann-fields} 
	$\vHarmv{}{\emptyset}{\Sp_{\rhat}}(\cU_{\rhat})
	=\zr(\cU(\rhat))\cap{}_{0}\dgen{}{}{\Sp_{\rhat}}(\cU(\rhat))=\{0\}$ are trivial 
	( the dimension is determined by the number of handles of $\cU_{\rhat}$, cf. 
	\cite{milani_decomposition_1988, picard_boundary_1982}\,) 
	Lemma \ref{lem:stp_toolbox-reproduced} yields
	\begin{align*}
		\zr(\cU_{\rhat})
			=\nabla\homo(\cU_{\rhat})
			 \oplus
			 \vHarmv{}{\emptyset}{\Sp_{\rhat}}(\cU_{\rhat})
			=\nabla\homo(\cU_{\rhat})\,.
	\end{align*}
	Thus, there exists $w\in\homo(\cU_{\rhat})$ such that $H=\nabla w$ in 
	$\cU_{\rhat}$. Using a suitable extension operator (\,e.g., the one of Stein\,), 
	we extend $w$ to $\widehat{w}\in\hgen{}{1}{-1,\Ga}(\Om)$. Then 
	$H-\nabla\widehat{w}\in\zrztom$ with $H-\nabla\widehat{w}=0$ in 
	$\cU_{\rhat}$ and hence
	\begin{align}\label{equ:app_proof-dir-neu_step-1-ext}
		H-\nabla\widehat{w}\in{}_{0}\rgen{}{}{\Ga_{1,\rhat}}(\Omrhat)
		\,,\quad\quad\Ga_{1,\rhat}=\Gat\cup\Sp_{\rhat}\,.
	\end{align}  
	From \eqref{equ:app_proof-dir-neu_step-1-ext} and 
	\eqref{equ:app_proof-dir-neu_step-1-bd} we conclude
	\begin{align*}
		\norm{H}_{\lt(\Om)}^2	
			&=\scp{H}
			      {H-\nabla\widehat{w}}_{\ltom}
			  +\scp{H}
			       {\nabla\widehat{w}}_{\ltom}\\
			&=\scp{\rot E}
				  {H-\nabla\widehat{w}}_{\lt(\Omrhat)}
			  -\underbrace{\scp{\div H}
			       {\widehat{w}}_{\ltom}	}_{=\,0}
			 =\scp{E}
			      {\rot(H-\nabla\widehat{w})}_{\lt(\Omrhat)}=0\,.
	\end{align*}
\end{proof}
\begin{lem}[Step 3]	
	Let $\frakB(\Omrhat)$ be a basis of $\vHarmv{}{\Ga_{1,\rhat}}{\Gan}(\Omrhat)$ 
	and let $\BGatom$ be defined as above. It holds
	\begin{align*}
		\vHarmtn{}(\Om)\cap\BGatom^{\perp}=\{0\}\,.
	\end{align*}
\end{lem}
\begin{proof}
	Let 
	$H\in\vHarmtn{}(\Om)\cap\BGatom^{\perp}$. 
	Then, for all $B\in\frakB(\Omrhat)$ we have by definition of $\BGatom$
	\begin{align*}
		\scp{\scrR_{\Omrhat}(H)}{B}_{\lt(\Omrhat)}
			=\scp{H}{\scrE_{\Om}(B)}_{\lt(\Om)}=0\,,
	\end{align*}
 	and hence by \eqref{equ:app_proof-dir-neu_zdznom-decomp}
 	\begin{align*}
 		\scrR_{\Omrhat}(H)\in
 		\zdzn(\Omrhat)
		\cap
 		\vHarmv{}{\Ga_{1,\rhat}}{\Gan}(\Omrhat)^{\perp}
	 	=\rot\Rzn(\Omrhat)\,.
	\end{align*}
	The assertion (\,$H=0$\,) now follows by continuing as in the latter proof 
	after \eqref{equ:app_proof-dir-neu_step-1-bd}.
\end{proof}
\end{document}

%% file: header_Osterbrink.tex
\usepackage[mathscr]{eucal}
\usepackage[english]{babel}
\usepackage{exscale,ifthen}
\usepackage{amsmath,amsfonts,amssymb,amsthm,amscd,amsbsy}
\usepackage{bbm,nicefrac,mathtools}
\usepackage{enumitem}
\usepackage{tikz-cd}
\usepackage{tikz,graphicx,color}
\usepackage{caption}[2008/08/24]
\usepackage{fancyhdr,array}
\usepackage{oldgerm}

\DeclareFontFamily{U}{mathx}{\hyphenchar\font45}
\DeclareFontShape{U}{mathx}{m}{n}{
      <5> <6> <7> <8> <9> <10>
      <10.95> <12> <14.4> <17.28> <20.74> <24.88>
      mathx10
      }{}
\DeclareSymbolFont{mathx}{U}{mathx}{m}{n}
\DeclareFontSubstitution{U}{mathx}{m}{n}
\DeclareMathAccent{\widecheck}{0}{mathx}{"71}
\DeclareMathAccent{\wideparen}{0}{mathx}{"75}

\newcommand{\preprintudemath}[5]{
\thispagestyle{empty}
\Large
\begin{center}SCHRIFTENREIHE DER FAKULT\"AT F\"UR MATHEMATIK\end{center}
\vspace*{5mm}
\begin{center}#1\end{center}
\vspace*{5mm}
\begin{center}by\end{center}
\begin{center}#2\end{center}
\vspace*{5mm}
\begin{center}#3\hspace{80mm}#4\end{center}
\newpage
\thispagestyle{empty}
\vspace*{210mm}
Received: #5
\newpage
\addtocounter{page}{-2}
\normalsize}

\ifthenelse{\equal{\pagesizeextendednormal}{extended}}
{\setlength{\textwidth}{16cm}\setlength{\textheight}{22.5cm}
\setlength{\oddsidemargin}{-0.2cm}
\setlength{\evensidemargin}{-0.2cm}}{}
\ifthenelse{\equal{\numberingequationsectionyesno}{yes}}
		   {\numberwithin{equation}{section}}{}

\makeatletter
\newcommand{\leqnomode}{\tagsleft@true}
\newcommand{\reqnomode}{\tagsleft@false}
\makeatother

\newcommand{\dsp}{\displaystyle}
\newcommand{\ovl}[1]{\overline{#1}}

\newcommand{\set}[2]{\{#1\,:\,#2\}}
\newcommand{\setb}[2]{\big\{#1\,:\,#2\big\}}
\newcommand{\setB}[2]{\Big\{#1\,:\,#2\Big\}}

\newcolumntype{L}[1]{>{\raggedright\arraybackslash}p{#1}}
\newcolumntype{C}[1]{>{\centering\arraybackslash}p{#1}}
\newcolumntype{R}[1]{>{\raggedleft\arraybackslash}p{#1}} 

\ifthenelse{\equal{\numberingtheoremsectionyesno}{yes}}
{\newtheorem{lem}{Lemma}[section]}
{\newtheorem{lem}{Lemma}}
\newtheorem{defi}[lem]{Definition}
\newtheorem{theo}[lem]{Theorem}
\newtheorem{cor}[lem]{Corollary}
\newtheorem{rem}[lem]{Remark}
\newtheorem{pro}[lem]{Proposition}

\newtheorem{genass}[lem]{General Assumption}

\newcommand{\al}{\alpha}
\newcommand{\da}{\delta}
\newcommand{\eps}{\varepsilon}

\newcommand{\ga}{\gamma}
\newcommand{\ka}{\kappa}

\newcommand{\om}{\omega}

\newcommand{\Ga}{\Gamma}
\newcommand{\La}{\mathrm{\Lambda}}
\newcommand{\Laz}{\La_{0}}
\newcommand{\tLaz}{\widetilde{\La}_{0}}\newcommand{\Om}{\Omega}
\newcommand{\Omb}{\ovl{\Om}}

\newcommand{\calD}{{\mathcal D}}
\newcommand{\calE}{{\mathcal E}}

\newcommand{\calH}{{\mathcal H}}

\newcommand{\calL}{{\mathcal L}}

\newcommand{\calN}{{\mathcal N}}
\newcommand{\calO}{{\mathcal O}}

\newcommand{\calR}{{\mathcal R}}

\newcommand{\sfC}{{\mathsf C}}
\newcommand{\sfD}{{\mathsf D}}

\newcommand{\sfH}{{\mathsf H}}

\newcommand{\sfL}{{\mathsf L}}

\newcommand{\sfR}{{\mathsf R}}

\newcommand{\sfV}{{\mathsf V}}

\newcommand{\scrB}{\mathscr{B}}

\newcommand{\scrE}{\mathscr{E}}

\newcommand{\scrM}{\mathscr{M}}

\newcommand{\scrR}{\mathscr{R}}

\newcommand{\frakB}{\mathfrak{B}}

\newcommand{\aLR}{
	\font\eqfont=cmr10 scaled 1200
	\font\arrfont=cmsy10 scaled 1200
	\textfont0=\eqfont
	\textfont2=\arrfont
	\protect\Relbar\protect\joinrel\Rightarrow}
\renewcommand{\Longrightarrow}{\aLR}
\renewcommand{\to}{\rightarrow}

\newcommand{\To}{\longrightarrow}

\newcommand{\restr}[2]{\left.#1 \right|_{#2}}
\newcommand{\map}[3]{#1:#2\longrightarrow#3}
\newcommand{\maps}[5]{#1:#2\;\longrightarrow\;#3\;,\;\;#4\;\longmapsto\;#5}
\newcommand{\Map}[5]	
    {\begin{array}{ccccc}	
    #1 & : & #2 & \longrightarrow & #3 \\[5pt]
    {}   & {}  & #4 & \longmapsto & #5	
    \end{array}}

\DeclareMathOperator{\supp}{supp}
\DeclareMathOperator{\dist}{dist}

\renewcommand{\Im}{\mathrm{Im}\,}

\newcommand{\adj}[1]{\overset{}{#1}^{\hspace*{-0.05cm}\ast}}

\newcommand{\A}{\mathrm{A}}
\newcommand{\M}{\mathrm{M}}

\newcommand{\G}{\mathrm{G}}

\newcommand{\p}{\partial}

\DeclareMathOperator{\grad}{grad}

\DeclareMathOperator{\rot}{rot}

\DeclareMathOperator{\divergenz}{div}
\renewcommand{\div}{\divergenz}

\newcommand{\sm}{\setminus}

\newcommand{\rtil}{\tilde{r}}

\newcommand{\rhat}{\hat{r}}

\newcommand{\ttil}{\tilde{t}}

\newcommand{\Gat}{\Ga_{1}}
\newcommand{\Gan}{\Ga_{2}}

\newcommand{\B}{\mathrm{B}}

\newcommand{\C}{\mathbb{C}}

\newcommand{\K}{\mathrm{K}}
\newcommand{\N}{\mathbb{N}}
\newcommand{\reals}{\mathbb{R}}
\newcommand{\U}{\mathrm{U}}
\newcommand{\cU}{\widecheck{\mathrm{U}}}
\newcommand{\Sp}{\mathrm{S}}
\newcommand{\dod}{\mathcal{D}}
\newcommand{\rg}{\mathcal{R}}
\renewcommand{\ker}{\mathcal{N}}
\newcommand{\V}{\mathrm{V}}

\newcommand{\gk}[1]{\mathcal{N}_{\mathsf{gen}}(#1)}
\newcommand{\gs}{\sigma_{\mathsf{gen}}(\scrM)}	

\newcommand{\rthree}{\reals^{3}}

\newcommand{\rttt}{\reals^{3\times3}}

\newcommand{\Omrhat}{\Om_{\rhat}}
\newcommand{\Omda}{\Om_{\da}}

\newcommand{\BGatom}{\frakB_{1}(\Om)}
\newcommand{\BGanom}{\frakB_{2}(\Om)}

\newcommand{\cgen}[3]{\overset{#1}{\sfC}{}^{#2}_{#3}}
\newcommand{\Cgen}[2]{\mathring{\sfC}{}^{#1}_{#2}}

\newcommand{\co}{\cgen{}{1}{}}

\newcommand{\ci}{\cgen{}{\substack{\infty\\[-5pt]\phantom{\infty}}}{}}

\newcommand{\cirthree}{\ci(\rthree)}

\newcommand{\cic}{\Cgen{\substack{\infty\\[-5pt]\phantom{\infty}}}{}}

\newcommand{\cicg}{\cgen{}{\substack{\infty\\[-5pt]\phantom{\infty}}}{\Ga}}
\newcommand{\cict}{\cgen{}{\substack{\infty\\[-5pt]\phantom{\infty}}}{\Gat}}
\newcommand{\cicn}{\cgen{}{\substack{\infty\\[-5pt]\phantom{\infty}}}{\Gan}}

\newcommand{\cicgom}{\cicg(\Om)}
\newcommand{\cictom}{\cict(\Om)}
\newcommand{\cicnom}{\cicn(\Om)}

\newcommand{\lsymb}{\sfL}
\newcommand{\lgen}[3]{\overset{#1}{\lsymb}{}^{#2}_{#3}}

\newcommand{\ltloc}{\lgen{}{2}{\mathrm{loc}}}

\newcommand{\lo}{\lgen{}{1}{}}
\newcommand{\lt}{\lgen{}{2}{}}
\newcommand{\li}{\lsymb^{\infty}}

\newcommand{\ltlocom}{\ltloc(\Om)}

\newcommand{\ltom}{\lt(\Om)}

\newcommand{\ltlocomb}{\ltloc(\Omb)}

\newcommand{\ltmo}{\lgen{}{2}{-1}}

\newcommand{\lts}{\lgen{}{2}{s}}
\newcommand{\ltspo}{\lgen{}{2}{s+1}}

\newcommand{\ltt}{\lgen{}{2}{t}}
\newcommand{\lttpo}{\lgen{}{2}{t+1}}

\newcommand{\ltmoom}{\ltmo(\Om)}

\newcommand{\ltsom}{\lts(\Om)}

\newcommand{\lttom}{\ltt(\Om)}
\newcommand{\lttpoom}{\lttpo(\Om)}

\newcommand{\ltmomka}{\lgen{}{2}{-1-\ka}}

\newcommand{\ltttil}{\lgen{}{2}{\ttil}}
\newcommand{\lttpal}{\lgen{}{2}{t+|\al|}}

\newcommand{\ltsmka}{\lgen{}{2}{s-\ka}}

\newcommand{\ltbig}[1]{\lgen{}{2}{>#1}}
\newcommand{\ltsm}[1]{\lgen{}{2}{<#1}}
\newcommand{\ltLa}{\lgen{}{2}{\La}}
\newcommand{\ltga}{\lgen{}{2}{\ga}}
\newcommand{\lteps}{\lgen{}{2}{\eps}}
\newcommand{\ltmu}{\lgen{}{2}{\mu}}

\newcommand{\ltmoeps}{\lgen{}{2}{-1,\eps}}
\newcommand{\lttga}{\lgen{}{2}{t,\ga}}

\newcommand{\ltmomkaom}{\ltmomka(\Om)}

\newcommand{\ltttilom}{\ltttil(\Om)}
\newcommand{\lttpalom}{\lttpal(\Om)}

\newcommand{\ltbigom}[1]{\ltbig{#1}(\Om)}
\newcommand{\ltsmom}[1]{\ltsm{#1}(\Om)}
\newcommand{\ltLaom}{\ltLa(\Om)}
\newcommand{\ltgaom}{\ltga(\Om)}
\newcommand{\ltepsom}{\lteps(\Om)}
\newcommand{\ltmuom}{\ltmu(\Om)}

\newcommand{\hsymb}{\sfH}
\newcommand{\hgen}[3]{\overset{#1}{\hsymb}{}^{#2}_{#3}}

\newcommand{\ho}{\hgen{}{1}{}}

\newcommand{\homo}{\hgen{}{1}{-1}}
\newcommand{\homog}{\hgen{}{1}{-1,\Ga}}
\newcommand{\homot}{\hgen{}{1}{-1,\Gat}}

\newcommand{\homoom}{\homo(\Om)}
\newcommand{\homotom}{\homot(\Om)}

\newcommand{\hos}{\hgen{}{1}{s}}

\newcommand{\hot}{\hgen{}{1}{t}}

\newcommand{\hott}{\hgen{}{1}{t,\Gat}}

\newcommand{\hmt}{\hgen{}{m}{t}}

\newcommand{\hmtn}{\hgen{}{m}{t,\Gan}}

\newcommand{\hotom}{\hot(\Om)}

\newcommand{\hottom}{\hott(\Om)}

\newcommand{\hmtom}{\hmt(\Om)}

\newcommand{\hmtnom}{\hmtn(\Om)}

\newcommand{\homomka}{\hgen{}{1}{-1-\ka}}

\newcommand{\Hsymb}{\text{\sf\bfseries H}}
\newcommand{\Hgen}[3]{\overset{#1}{\Hsymb}{}^{#2}_{#3}}

\newcommand{\Hozt}{\Hgen{}{1}{\Gat}}

\newcommand{\Htwo}{\Hgen{}{2}{}}

\newcommand{\Hm}{\Hgen{}{m}{}}

\newcommand{\Hoztom}{\Hozt(\Om)}

\newcommand{\Hmom}{\Hm(\Om)}

\newcommand{\Hos}{\Hgen{}{1}{s}}

\newcommand{\Hot}{\Hgen{}{1}{t}}

\newcommand{\Hott}{\Hgen{}{1}{t,\Gat}}

\newcommand{\Hmt}{\Hgen{}{m}{t}}

\newcommand{\Hotom}{\Hot(\Om)}

\newcommand{\Hottom}{\Hott(\Om)}

\newcommand{\Hmtom}{\Hmt(\Om)}

\newcommand{\rsymb}{\sfR}
\newcommand{\rgen}[3]{\overset{#1}{\rsymb}{}^{#2}_{#3}}

\newcommand{\rloct}{\rgen{}{}{\mathrm{loc},\Gat}}
\newcommand{\rlocn}{\rgen{}{}{\mathrm{loc},\Gan}}

\newcommand{\rloctom}{\rloct(\Om)}
\newcommand{\rlocnom}{\rlocn(\Om)}

\newcommand{\zr}{{}_{0}\rgen{}{}{}}

\newcommand{\zrzt}{{}_{0}\rgen{}{}{\Gat}}
\newcommand{\zrzn}{{}_{0}\rgen{}{}{\Gan}}

\newcommand{\zrztom}{\zrzt(\Om)}
\newcommand{\zrznom}{\zrzn(\Om)}

\newcommand{\rmo}{\rgen{}{}{-1}}

\newcommand{\rmot}{\rgen{}{}{-1,\Gat}}
\newcommand{\rmon}{\rgen{}{}{-1,\Gan}}

\newcommand{\rmoom}{\rmo(\Om)}

\newcommand{\rmotom}{\rmot(\Om)}
\newcommand{\rmonom}{\rmon(\Om)}

\newcommand{\zrmot}{{}_{0}\rgen{}{}{-1,\Gat}}

\newcommand{\zrmotom}{\zrmot(\Om)}

\newcommand{\rsmon}{\rgen{}{}{s-1,\Gan}}

\newcommand{\rsmonom}{\rsmon(\Om)}

\newcommand{\zrst}{{}_{0}\rgen{}{}{s,\Gat}}

\newcommand{\zrstom}{\zrst(\Om)}

\newcommand{\rt}{\rgen{}{}{t}}

\newcommand{\rtt}{\rgen{}{}{t,\Gat}}

\newcommand{\rtom}{\rt(\Om)}

\newcommand{\rttom}{\rtt(\Om)}

\newcommand{\zrt}{{}_{0}\rgen{}{}{t}}

\newcommand{\zrtt}{{}_{0}\rgen{}{}{t,\Gat}}

\newcommand{\zrtom}{\zrt(\Om)}

\newcommand{\zrttom}{\zrtt(\Om)}

\newcommand{\rbigt}[1]{\rgen{}{}{>#1,\Gat}}
\newcommand{\rbign}[1]{\rgen{}{}{>#1,\Gan}}

\newcommand{\Rsymb}{\text{\sf\bfseries R}}
\newcommand{\Rgen}[3]{\overset{#1}{\Rsymb}{}^{#2}_{#3}}

\newcommand{\R}{\Rgen{}{}{}}

\newcommand{\Rzt}{\Rgen{}{}{\Gat}}
\newcommand{\Rzn}{\Rgen{}{}{\Gan}}

\newcommand{\Rom}{\R(\Om)}

\newcommand{\Rztom}{\Rzt(\Om)}
\newcommand{\Rznom}{\Rzn(\Om)}

\newcommand{\Rmot}{\Rgen{}{}{-1,\Gat}}
\newcommand{\Rmon}{\Rgen{}{}{-1,\Gan}}

\newcommand{\Rmotom}{\Rmot(\Om)}
\newcommand{\Rmonom}{\Rmon(\Om)}

\newcommand{\Rst}{\Rgen{}{}{s,\Gat}}
\newcommand{\Rsn}{\Rgen{}{}{s,\Gan}}

\newcommand{\Rstom}{\Rst(\Om)}
\newcommand{\Rsnom}{\Rsn(\Om)}

\newcommand{\Rt}{\Rgen{}{}{t}}

\newcommand{\Rtt}{\Rgen{}{}{t,\Gat}}
\newcommand{\Rtn}{\Rgen{}{}{t,\Gan}}

\newcommand{\Rtom}{\Rt(\Om)}

\newcommand{\Rttom}{\Rtt(\Om)}
\newcommand{\Rtnom}{\Rtn(\Om)}

\newcommand{\Rttilt}{\Rgen{}{}{\ttil,\Gat}}
\newcommand{\Rttiln}{\Rgen{}{}{\ttil,\Gan}}

\newcommand{\Rsm}[1]{\Rgen{}{}{<#1}}

\newcommand{\Rsmt}[1]{\Rgen{}{}{<#1,\Gat}}
\newcommand{\Rsmn}[1]{\Rgen{}{}{<#1,\Gan}}

\newcommand{\Rttiltom}{\Rttilt(\Om)}
\newcommand{\Rttilnom}{\Rttiln(\Om)}

\newcommand{\Rsmtom}[1]{\Rsmt{#1}(\Om)}
\newcommand{\Rsmnom}[1]{\Rsmn{#1}(\Om)}

\newcommand{\dsymb}{\sfD}
\newcommand{\dgen}[3]{\overset{#1}{\dsymb}{}^{#2}_{#3}}

\newcommand{\zd}{{}_{0}\dgen{}{}{}}

\newcommand{\zdzt}{{}_{0}\dgen{}{}{\Gat}}
\newcommand{\zdzn}{{}_{0}\dgen{}{}{\Gan}}

\newcommand{\zdztom}{\zdzt(\Om)}
\newcommand{\zdznom}{\zdzn(\Om)}

\newcommand{\dmo}{\dgen{}{}{-1}}

\newcommand{\dmot}{\dgen{}{}{-1,\Gat}}
\newcommand{\dmon}{\dgen{}{}{-1,\Gan}}

\newcommand{\dmoom}{\dmo(\Om)}

\newcommand{\dmotom}{\dmot(\Om)}
\newcommand{\dmonom}{\dmon(\Om)}

\newcommand{\zdmon}{{}_{0}\dgen{}{}{-1,\Gan}}

\newcommand{\zdmonom}{\zdmon(\Om)}

\newcommand{\zdsn}{{}_{0}\dgen{}{}{s,\Gan}}

\newcommand{\zdsnom}{\zdsn(\Om)}

\newcommand{\dt}{\dgen{}{}{t}}

\newcommand{\dtt}{\dgen{}{}{t,\Gat}}

\newcommand{\dtom}{\dt(\Om)}

\newcommand{\dttom}{\dtt(\Om)}

\newcommand{\zdt}{{}_{0}\dgen{}{}{t}}

\newcommand{\zdtt}{{}_{0}\dgen{}{}{t,\Gat}}
\newcommand{\zdtn}{{}_{0}\dgen{}{}{t,\Gan}}

\newcommand{\zdtom}{\zdt(\Om)}

\newcommand{\zdttom}{\zdtt(\Om)}
\newcommand{\zdtnom}{\zdtn(\Om)}

\newcommand{\dbigt}[1]{\dgen{}{}{>#1,\Gat}}
\newcommand{\dbign}[1]{\dgen{}{}{>#1,\Gan}}

\newcommand{\zdttilt}{{}_{0}\dgen{}{}{\ttil,\Gat}}
\newcommand{\zdttiln}{{}_{0}\dgen{}{}{\ttil,\Gan}}

\newcommand{\Dsymb}{\text{\sf\bfseries D}}
\newcommand{\Dgen}[3]{\overset{#1}{\Dsymb}{}^{#2}_{#3}}

\newcommand{\D}{\Dgen{}{}{}}

\newcommand{\Dzt}{\Dgen{}{}{\Gat}}
\newcommand{\Dzn}{\Dgen{}{}{\Gan}}

\newcommand{\Dom}{\D(\Om)}

\newcommand{\Dztom}{\Dzt(\Om)}
\newcommand{\Dznom}{\Dzn(\Om)}

\newcommand{\Dmot}{\Dgen{}{}{-1,\Gat}}
\newcommand{\Dmon}{\Dgen{}{}{-1,\Gan}}

\newcommand{\Dmotom}{\Dmot(\Om)}
\newcommand{\Dmonom}{\Dmon(\Om)}

\newcommand{\Ds}{\Dgen{}{}{s}}

\newcommand{\Dst}{\Dgen{}{}{s,\Gat}}
\newcommand{\Dsn}{\Dgen{}{}{s,\Gan}}

\newcommand{\Dstom}{\Dst(\Om)}
\newcommand{\Dsnom}{\Dsn(\Om)}

\newcommand{\Dt}{\Dgen{}{}{t}}

\newcommand{\Dtt}{\Dgen{}{}{t,\Gat}}
\newcommand{\Dtn}{\Dgen{}{}{t,\Gan}}

\newcommand{\Dtom}{\Dt(\Om)}

\newcommand{\Dttom}{\Dtt(\Om)}
\newcommand{\Dtnom}{\Dtn(\Om)}

\newcommand{\Dttiltom}{\Dgen{}{}{\ttil,\Gat}(\Om)}
\newcommand{\Dttilnom}{\Dgen{}{}{\ttil,\Gan}(\Om)}

\newcommand{\harmsymb}{\mathscr{H}}

\newcommand{\harmgen}[4]{{}_{#1}{\harmsymb}{}_{#2,#3,#4}}
\newcommand{\Harmgen}[3]{{}_{#1}{\harmsymb}{}_{#2,#3}}

\newcommand{\vHarmv}[3]{\Harmgen{#1}{#2}{#3}}

\newcommand{\vHarmtn}[1]{\Harmgen{#1}{\Gat}{\Gan}}
\newcommand{\vHarmnt}[1]{\Harmgen{#1}{\Gan}{\Gat}}

\newcommand{\gharms}[2]{\harmgen{\ga}{s}{#1}{#2}}

\newcommand{\gharmstn}{\gharms{\Gat}{\Gan}}

\newcommand{\vharmmotn}[1]{\harmgen{#1}{-1}{\Gat}{\Gan}}
\newcommand{\vharmmont}[1]{\harmgen{#1}{-1}{\Gan}{\Gat}}

\newcommand{\gharmstnom}{\gharmstn(\Om)}

\newcommand{\norm}[1]{\left\|\,#1\,\right\|}

\newcommand{\normltom}[1]{\norm{#1}_{\ltom}}

\newcommand{\normltmoom}[1]{\norm{#1}_{\ltmoom}}

\newcommand{\normltsom}[1]{\norm{#1}_{\ltsom}}

\newcommand{\normlttom}[1]{\norm{#1}_{\lttom}}

\newcommand{\scp}[2]{\langle\,#1\,,#2\,\rangle}

\newcommand{\scpltom}[2]{\scp{#1}{#2}_{\ltom}}

\newcommand{\scpltLaom}[2]{\scp{#1}{#2}_{\ltLaom}}

\newcommand{\pthreevec}[3]{\begin{pmatrix}#1\\#2\\#3\end{pmatrix}}

\newcommand{\ptwomat}[4]{\begin{pmatrix}#1&#2\\#3&#4\end{pmatrix}}